\documentclass[12pt,reqno]{amsart}

\usepackage{amssymb,longtable}

\newcommand{\e}{\varepsilon}

\newcommand{\LL}{\mathcal{L}}
\newcommand{\M}{\mathcal{M}}
\newcommand{\la}{\lambda}
\newcommand{\al}{\alpha}

\newcommand{\fy}{\varphi}

\newcommand{\p}{\partial}
\newcommand{\I}{\infty}
\newcommand{\wt}[1]{\widetilde {#1}}
\newcommand{\ti}{\tilde}
\newcommand{\R}{\mathbb{R}}

\newcommand{\C}{\mathbb{C}}

\newcommand{\Z}{\mathbb{Z}}

\newcommand{\cS}{\mathcal{S}}
\renewcommand{\Re}{\mathop{\mathrm{Re}}}
\renewcommand{\Im}{\mathop{\mathrm{Im}}}
\renewcommand{\bar}{\overline}

\newcommand{\UU}{\mathcal{U}}
\newcommand{\Sg}{\mathfrak{S}}

\numberwithin{equation}{section}

\newtheorem{thm}{Theorem}[section]
\newtheorem{cor}[thm]{Corollary}
\newtheorem{lem}[thm]{Lemma}
\newtheorem{prop}[thm]{Proposition}

\theoremstyle{remark}
\newtheorem{rem}{Remark}[section]
\newtheorem{defn}{Definition}[section]

\newcommand{\ran}{\rangle}
\newcommand{\lan}{\langle}

\newcommand{\lec}{\lesssim}
\newcommand{\gec}{\gtrsim}

\newcommand{\EQ}[1]{\begin{equation} \begin{split} #1 \end{split} \end{equation}}
\newcommand{\Del}[1]{}
\newcommand{\CAS}[1]{\begin{cases} #1 \end{cases}}
\newcommand{\pt}{&}
\newcommand{\pr}{\\ &}
\newcommand{\pq}{\quad}
\newcommand{\pn}{}

\newcommand{\prQ}{\\ &\qquad}

\newcommand{\prQQ}{\\ &\qquad\qquad}

\newcommand{\LR}[1]{{\lan #1 \ran}}
\newcommand{\de}{\delta}
\newcommand{\si}{\sigma}

\newcommand{\be}{\beta}
\newcommand{\ka}{\kappa}
\newcommand{\ga}{\gamma}

\newcommand{\na}{\nabla}

\renewcommand{\th}{\theta}

\newcommand{\B}{\mathcal{B}}
\newcommand{\De}{\Delta}
\newcommand{\Om}{\Omega}
\newcommand{\Ga}{\Gamma}

\newcommand{\Cu}{\bigcup}

\renewcommand{\H}{\mathcal{H}}

\newcommand{\GZ}{\mathfrak{G}}
\newcommand{\sg}{\mathfrak{s}}
\newcommand{\HH}{\mathcal{H}}

\def\dist{\mathrm{dist}}
\def\calH{{\mathcal H}}
\def\Rw{\binom{w}{\bar{w}}}
\def\Rv{\binom{v}{\bar{v}}}
\def\nn{\nonumber}
\def\Wdisp{W_{\mathrm{disp}}}
\def\Wroot{W_{\mathrm{root}}}
\def\Preal{P_{\mathrm{real}}}
\def\tor{{\mathbb T}}

\def\eps{\varepsilon}
\def\sign{\mathrm{sign}}

\def\const{\mathrm{const}}

\def\cN{{\mathcal N}}
\def\err{{\mathrm{err}}}
\def\T{{\mathcal T}}
\def\Id{\mathrm{Id}}

\author{K.~Nakanishi}
\address{Department of Mathematics, Kyoto University\\ Kyoto 606-8502, Japan}
\email{n-kenji@math.kyoto-u.ac.jp}

\author{W.~Schlag}
\address{Department of  Mathematics, The University of Chicago\\ Chicago, IL 60615, U.S.A.} 
\email{schlag@math.uchicago.edu}

\title[Global NLS dynamics above the ground state energy]{Global dynamics above the ground state energy\\ for the  cubic NLS equation in 3D}

\begin{document}




\begin{abstract}
We extend the result in \cite{NakS} on the nonlinear Klein-Gordon equation to the nonlinear Schr\"odinger equation
with the focusing cubic nonlinearity in three dimensions, for radial data of energy at most slightly above that of the ground state.
We prove that the initial data set splits into nine nonempty, pairwise disjoint regions which are characterized
by the distinct behaviors of the solution for large time: blow-up, scattering to $0$, or scattering to the family of ground states generated by the phase and scaling freedom. 
Solutions of this latter type form 
a smooth center-stable manifold, which contains the ground states
and separates the phase space locally into two connected regions exhibiting 
blow-up and scattering to $0$, respectively. 
The special solutions found by Duyckaerts, Roudenko~\cite{DuyR}, following the seminal work on threshold solutions by Duyckaerts, Merle~\cite{DM1},  appear here as the unique one-dimensional unstable/stable 
manifolds emanating from the ground states. In analogy with~\cite{NakS}, the proof combines the hyperbolic dynamics near the ground states with
the variational structure away from them. 
The main technical ingredient in the proof is a ``one-pass'' theorem which precludes ``almost homoclinic orbits", i.e., those solutions starting in, then moving away from, and finally returning to, a small neighborhood of the ground states. 
The main new difficulty compared with the Klein-Gordon case is the lack of finite propagation speed. We need the radial Sobolev inequality for the error estimate in the virial argument. 
Another major difference between \cite{NakS} and this paper is the need to control two modulation parameters. 
\end{abstract}

\maketitle

\tableofcontents

\section{Introduction}\label{sec:intro}

The local well-posedness of the cubic NLS equation
\EQ{\label{eq:Schr3}
i\p_{t} u - \Delta u = |u|^{2} u
}
in the energy space $H^{1}$ is classical, see Strauss~\cite{Strauss}, Sulem, Sulem~\cite{SS99}, Cazenave~\cite{Caz}, and Tao~\cite{Tao}.  
One has mass and energy conservation 
\EQ{\label{eq:MundE}
M(u) &= \frac12\| u\|_{2}^{2} = \const.\\
E(u) &= \frac12\|\nabla u\|_{2}^{2} - \frac14 \|u\|_{4}^{4}=\const.,
}
where $\|\cdot\|_p$ denotes the $L^p(\R^3)$ norm. 
Data with small $H^{1}$ 
norm have globally defined solutions which scatter to a free wave. 
In the defocusing case it is known that all energy solutions scatter to zero, see Ginibre, Velo~\cite{GVelo1}, \cite{GVelo2}. 
In contrast,  \eqref{eq:Schr3} is known to exhibit energy data for which the solutions blow up in finite time. In fact, Glassey~\cite{Glassey}
proved that all data of negative energy are of this type provided they also have finite variance. The latter assumption was later removed
in the radial case by Ogawa, Tsutsumi~\cite{OgTs}. 

Eq.~\eqref{eq:Schr3} possesses a family of special oscillatory solutions of the form $u(t,x)=e^{-it\alpha^{2}+i\th} Q(x,\alpha)$ where $\alpha>0$ and 
\[
-\Delta Q(\cdot,\alpha) + \al^{2} Q(\cdot,\alpha) = |Q|^{2}Q(\cdot,\alpha)
\]
There is a unique positive, radial solution to this equation called the ground state, see Strauss~\cite{Strauss77}, Berestycki, Lions~\cite{BerLions}, Coffman~\cite{Coff}, Kwong~\cite{Kwong}.  
It is characterized as the solution of minimal action. 
Letting modulation and Galilean symmetries act on these special solutions $u(t,x)$  generates an eight-dimensional manifold of solitons.
In the radial case, the manifold is only two-dimensional. 

The question of orbital stability of these solitons in the energy space was settled by Weinstein~\cite{Wein1}, \cite{Wein2}, Berestycki, Cazenave~\cite{BerCaz}, and Cazenave, Lions~\cite{CazLions}.
A general theory which covers this case was developed by Grillakis, Shatah, Strauss~\cite{GSS1}, \cite{GSS2}. 
The cut-off in the power $|u|^{p-1}u$ in the $n$-dimensional case turns out to be the $L^{2}$ critical one $p_{0}=\frac{4}{n}+1$, with $p\ge p_{0}$ being unstable and $p<p_{0}$ stable. 
In particular, the cubic NLS~\eqref{eq:Schr3} is unstable. 
Recently, Holmer, Roudenko~\cite{HolR} showed that for all radial solutions $u$ with mass $\|u\|_{2}=\|Q\|_{2}$ and energy $E(u)<E(Q)$ there is the following dichotomy: 
if $\|\nabla u\|_{2}<\|\nabla Q\|_{2}$ one has global existence and scattering (as $|t|\to\infty$), whereas for $\|\nabla u\|_{2}> \|\nabla Q\|_{2}$ there is finite time blowup in both
time directions. The radial assumption was then removed in Duyckaerts, Holmer, Roudenko~\cite{DHR}.  Note that the mass condition is easily removed by scaling, with $M(u)E(u)$ being the natural
scaling-invariant version of the energy, and with $M(u)\|\nabla u\|_{2}^2$ replacing $\|\nabla u\|_{2}^2$. It follows from the variational properties of~$Q$ that these regions are invariant
under the NLS flow. 
The methods in both papers follow the ideology of Kenig-Merle~\cite{KM1}, \cite{KM2} which in turn use the concentrated compactness
decompositions of Bahouri, Gerard~\cite{BaG}, Merle, Vega~\cite{MeV},  as well as  Keraani~\cite{Keraani}. 

In a different direction, in recent years several authors have studied {\em conditional asymptotic stability} for the case of~\eqref{eq:Schr3} as well as other equations, see~\cite{S}, \cite{KrS1}, and Beceanu~\cite{Bec1}. This refers to the fact that solitons remain asymptotically stable even in the unstable case provided the perturbations are chosen to lie
on a manifold of finite codimension near the soliton manifold. The number of ``missing'' dimensions here equals the number of exponentially unstable modes of the linearized
equation. In the case of NLS this number equals~$1$.  These investigations are related to the classical notion of stable, unstable, and center-stable manifolds in dynamical
systems, see Bates, Jones~\cite{BJones} and Gesztesy, Jones, Latushkin, Stanislavova~\cite{GJLS} for a development of these ideas applicable to  NLKG and NLS. 

In this paper we find that the center-stable manifolds act as boundary between a region of finite time blow-up and one of scattering to zero. 
In what follows
\[
 \cS_{\alpha}:=\{ e^{i\theta}Q(\cdot,\alpha)\mid  \theta\in\R\}, \pq \cS:=\Cu_{\al>0} \cS_\al, 
\]
and we set $Q=Q(\cdot,1)$ for convenience. Then $Q(x,\alpha)=\alpha Q(\alpha x)$ and $M(Q(\cdot,\alpha))=\alpha^{-1}M(Q)$. 
First, we present the following result which does not rely on the notion of a center-stable manifold. Let $\HH=H^{1}_{\mathrm{rad}}(\R^{3})$ and 
\EQ{\label{eq:HHe}
 \HH^\e:=\{u\in \HH \mid M(u)E(u)<M(Q)(E(Q)+\e^2)\}
}
as well as 
\EQ{\label{eq:HHea} 
 \HH^\e_{\alpha}:= \HH^\e\cap \{u\in\HH\mid M(u)= M(Q(\cdot,\alpha))\}
}
for any $\alpha>0$. 

\begin{thm}
\label{thm:main}   
There exists $\e>0$ small such that all solutions of \eqref{eq:Schr3} with data in~$\HH_{1}^{\e}$ exhibit one of the following nine
different scenarios, with each case being attained by infinitely many data in~$\HH^{\e}_{1}$: 
\begin{enumerate}
\item Scattering to $0$ for both  $t\to\pm\I$, 
\item Finite time blowup on both sides $\pm t>0$, 
\item Scattering to $0$ as $t\to\I$ and finite time blowup in $t<0$, 
\item Finite time blowup in $t>0$ and scattering to $0$ as $t\to-\I$, 
\item Trapped by $\cS_{1}$ for $t\to\I$ and scattering to $0$ as $t\to-\I$, 
\item Scattering to $0$ as $t\to\I$ and trapped by $\cS_{1}$ as $t\to-\I$, 
\item Trapped by $\cS_{1}$ for $t\to\I$ and finite time blowup in $t<0$, 
\item Finite time blowup in $t>0$ and trapped by $\cS_{1}$ as $t\to-\I$, 
\item Trapped by $\cS_{1}$  as $t\to\pm\I$, 
\end{enumerate}
where ``trapped by $\cS_{1}$" means that the solution stays in a $O(\e)$ neighborhood of $\cS_{1}$ relative to~$H^{1}$ forever after 
some time (or before some time). The initial data sets for (1)-(4), respectively, are open in~$\HH^{\e}_{1}$. The set of data in $H^{1}$ for 
which the associated solutions of~\eqref{eq:Schr3} forward scatter, 
i.e., $(1)\cup(3)\cup(6)$, is open, pathwise connected, and unbounded; in fact, it contains curves which connect $0$ to~$\I$
in~$H^{1}$. 
\end{thm}

The reason behind the number $9$ is simply that all combinations of the three possibilities at $t=+\I$ (blowup, scattering, trapping) and the corresponding ones
at $t=-\I$ are allowed. 
The theorem applies to solutions of any mass by rescaling. More precisely, if  $u\in \HH^{\e}_{\alpha}$, then the statement
remains intact with $\cS_{1}$ replaced by $\cS_{\alpha}$ and ``trapped'' by $\cS_{\alpha}$ now meaning that $\dist(u,\cS_{\alpha})\lec\e$ where the
distance is measured in the metric 
\EQ{\label{eq:alphmetric}
\| \cdot \|_{H^{1}_{\alpha}}:= \Big( \alpha^{-1} \| \cdot \|_{\dot H^{1}}^{2} + \alpha \| \cdot\|_{2}^{2}\Big)^{\frac12}.
}
As in~\cite{NakS}, the main novel ingredient is the ``one-pass theorem'', see Theorem~\ref{thm:onepass} below. It precludes almost homoclinic
orbits which start very close to~$\cS_{1}$ and eventually return very close to~$\cS_{1}$. In combination with an analysis of the hyperbolic dynamics 
near $\cS_{1}$ which results from the exponentially unstable nature of the ground state solution, this allows one to show that the fate of the solution is governed
by a virial-type functional $K$ after it exits a neighborhood of~$\cS_{1}$. 

Using some finer spectral properties
 of the  Hamiltonian obtained by linearizing the NLS equation around~$Q$, see Proposition~\ref{prop:spectral}, 
we can formulate the following stronger statement which describes in more detail what ``trapping'' means. In this case it is better not to freeze the mass. 
In other words, we work with the full set~$\HH^{\e}$. 
We require the following terminology:

\begin{defn}\label{def:scattoQ} Let $u(0)\in \HH^{\e}$ define a solution $u(t)$ of~\eqref{eq:Schr3} for all $t\ge0$. 
We say that {\em $u$ forward scatters to $\cS$} iff 
there exist continuous curves $\theta:[0,\I)\to\R$ and $\alpha:[0,\I)\to (0,\I)$, as well as $u_{\I}\in \HH$ such that
for all $t\ge0$ 
\EQ{\label{eq:scattoS}
u(t) = e^{i\theta(t)}Q(\cdot,\alpha(t)) + e^{-it\Delta} u_{\I} + \Omega(t)  
}
where $\|\Omega(t)\|_{H^{1}}\to0$ as $t\to\I$,  $\alpha(t)\to \alpha_{\infty}>0$ as $t\to\I$. 
\end{defn} 

Note that one then necessarily has 
\EQ{\label{eq:MEal}
M(u) &= M(Q(\cdot,\alpha_{\I})) + M(u_{\I}) = \alpha_{\I}^{-1} M( Q) +  M(u_{\I}) \\
E(u) &= E( Q(\cdot,\alpha_{\I})) + \frac12 \|\nabla u_{\I}\|_{2}^{2} = \alpha_{\I} E(Q) + \frac12 \| \nabla u_{\I} \|_{2}^{2}
}
whence (using that $E(Q)=M(Q)>0$), 
\EQ{
 \pt \alpha_{\I}^{-1} \| \nabla u_{\I}\|_{2}^{2} + \alpha_{\I} \| u_{\I} \|_{2}^{2} + \frac{\| u_{\I}\|_{2}^{2}  \| \nabla u_{\I} \|_{2}^{2}}{2M(Q)} \le 2\e^{2}, 
 \pq \frac{M(Q)}{M(u)} \le \al_\I \le \frac{E(u)}{E(Q)},}
and in particular, we conclude that  $\| u_{\I} \|_{H^{1}_{{{\alpha_{\I}}}}}\le\e$, that $\al_\I$ is bounded from both above and below, and that $M(u)E(u)\ge M(Q)E(Q)$. 

The heuristic meaning of~\eqref{eq:scattoS} is simply that $u$ asymptotically decomposes
into a soliton $e^{i\theta_\I(t)}Q(\cdot,\alpha_{\I})$ plus an $H^{1}$-solution to the free Schr\"odinger equation 
(however,  the phase $\theta_\I$ is not precisely the one associated with $Q(\cdot,\alpha_{\I})$ which would mean $-t\alpha_{\I}^{2}+\ga_{\I}$).
In fact, in those cases where we can
establish~\eqref{eq:scattoS} we will be able to obtain finer statements on $\theta$ and $\alpha$, cf.\ Section~\ref{sec:mf}. 

\begin{thm}\label{thm:main*} 
There exists $\e>0$ small such that all solutions of \eqref{eq:Schr3} with data in~$\HH^{\e}$ exhibit one of the  nine
different scenarios described in Theorem~\ref{thm:main}, provided we  replace ``trapped by $\cS_{1}$" with ``scattering to $\cS$". Moreover, 
each case is attained by infinitely many data in~$\HH^{\e}$. 
The sets $(5)\cup(7)\cup(9)$ and $(6)\cup(8)\cup(9)$ are smooth codimension-one manifolds in the phase space $\HH$. 
Similarly, (9) is a smooth manifold of codimension two, and it contains~$\cS$.
\end{thm}

Using common terminology from dynamical systems, see for example Hirsch, Pugh, Shub~\cite{HPS}, Vanderbauwhede~\cite{V}, and Bates, Jones~\cite{BJones}, we can say that 
$(5)\cup(7)\cup(9)$ and $(6)\cup(8)\cup(9)$ are the center-stable manifold~$\M_{cs}$, resp.~the center-unstable 
manifold $\M_{cu}$, associated with $Q$ --- {\em modulo the symmetries} given by $\alpha$ and $\theta$. 
Since center manifolds are in general not unique it might be more precise to say ``a center-stable manifold'' here. 
However, our manifolds are naturally unique for the global characterization in Theorem \ref{thm:main}.  Similarly, $(9)$ is  
the center manifold of~$Q$, again modulo the symmetries given by $\alpha$ and $\theta$. 

Every point $p\in \cS$ has a neighborhood $B_{\e}(p)$ of size $\lec\e$ relative to the metric~\eqref{eq:alphmetric} 
with $\alpha=M(Q)/M(p)$, such that $B_{\e}(p)$ is divided
by~$\M_{cs}$ into two connected components; all data in one component lead to finite time blow-up for positive times, 
whereas all data in the other lead to global solutions for positive times which scatter to zero as $t\to +\I$. All solutions 
starting on $\M_{cs}$ itself scatter to~$\cS$ in the sense of~\eqref{eq:scattoS} as $t\to+\I$.

The study of stable/unstable/center-stable manifolds near equillibria of ODEs (also in infinite dimensions) has a long 
history in dynamics. In fact, their existence for the cubic NLS~\eqref{eq:Schr3} was
shown in~\cite{GJLS} and~\cite{BJones}. However, in contrast to Theorem~\ref{thm:main*} no results are obtained there concerning the long-time behavior of
the solutions on the center manifold. The unique (up to the modulation and dilation symmetries) one-dimensional stable/unstable manifolds emanating from~$Q$ are characterized
by the requirement that $u(t)\to e^{-it}Q$ in $H^{1}$ exponentially fast as $t\to\I$ or $t\to -\I$. Clearly, the corresponding solutions must
have energy equal to that of~$Q$.  The same definition applies to $\cS$ with $Q$ being replaced by~$e^{-it\al^2+i\theta_0}Q(\cdot,\al)$. 
In our work these one-dimensional manifolds (up to the symmetries) appear naturally in the 
form of those solutions found by Duyckaerts, Roudenko~\cite{DuyR}. It is important to note that we can therefore completely describe the global (i.e., both as $t\to\I$ as well
as $t\to-\I$)
behavior of the stable/unstable manifolds in this setting. 

\begin{thm} \label{thm:threshold}
Consider the limiting case $\e\to 0$ in Theorem \ref{thm:main}, i.e., all the radial solutions satisfying $E(u) \le E(Q)$ and $M(u)=M(Q)$. 
Then the sets (3) and (4) vanish, while the sets (5)-(9) are characterized, with some special solutions $W_\pm$ of~\eqref{eq:Schr3}, as follows:
\EQ{
 \pt (5)=\{e^{i\theta}W_-(t-t_0) \mid t_0,\theta \in\R\},
 \pq (6)=\{ e^{i\theta} \overline{W_-}(-t-t_0) \mid t_0,\theta \in\R\},
 \pr (7)=\{  e^{i\theta}W_+(t-t_0) \mid t_0,\theta \in\R\},
 \pq (8)=\{  e^{i\theta} \overline{W_+}(-t-t_0) \mid t_0,\theta \in\R\},
 \pr (9)=\{  e^{-i(t+\theta)} Q  \mid \theta \in\R\}.}
 The sets
 $(5)\cup(7)\cup(9)$ form the stable manifold, whereas $ (6)\cup(8)\cup(9)$ are the unstable manifold of~$Q$, up to the modulation symmetry.
In other words, solutions in $(5), (7)$ and $(6), (8)$ approach a soliton trajectory in $\cS_1$ exponentially fast as $t\to\I$ or $t\to-\I$, respectively. 
 An analogous statement holds without the mass constraint, but then these sets take the form 
 $\{ e^{i \theta}\alpha W_{\pm}(\al^{2}(t-t_{0}),\alpha x)\}$, $\{ e^{i \theta}\alpha \overline{W_{\pm}}(-\al^{2}(t+t_{0}),\alpha x)\}$, resp.\ $\{e^{-i(t\al^2+\theta)} Q(\cdot,\alpha)\}$,  
 where $\theta,\alpha$ vary. 
\end{thm}

This paper is organized as follows. In Section~\ref{sec:ground} we review some variational properties of the ground state and discuss
the linearized operators. In Section~\ref{sec:mod} we present the modulation method which we use in the proof of Theorem~\ref{thm:main}.
Since Theorem~\ref{thm:main} is closer to orbital stability than asymptotic stability, the modulation approach  of Section~\ref{sec:mod}  
is less precise but easier to work with than the one usually 
employed in asymptotic stability theory. 
Section~\ref{sec:onepass} presents the one-pass theorem,
and is of central importance to the entire paper. 
The proof of that theorem is more involved than in the Klein-Gordon case~\cite{NakS}, due to the lack of finite speed of propagation. 
We will modify Ogawa-Tsutsumi's {\it saturated virial} identity \cite{OgTs} in the radial energy space, fitting it in the variational argument away from the ground state. 
Section~\ref{sec:scat} shows by a Kenig-Merle type argument~\cite{KM1}, that those solutions which are guaranteed by the one-pass theorem to
exist for all positive times actually scatter to zero. The proof of Theorem~\ref{thm:main} is then given in Section~\ref{sec:proofmain}. 
Up until that point, our arguments do not require any fine spectral properties of the linearized NLS Hamiltonian. This changes in Section~\ref{sec:mf} where
we construct the center-stable manifold in the radial energy class near~$Q$ following the method in~\cite{S} and~\cite{Bec2} (we remark that Beceanu~\cite{Bec2} has constructed
the manifold in $\dot{H}^{\frac12}$ without any radial assumption). Some of the aforementioned spectral properties are  -- at least for the moment -- 
only known via numerically assisted arguments, see~\cite{DS} and~\cite{MS} as well as 
Proposition~\ref{prop:spectral}. More precisely,  for the structure of the real spectrum we rely on the recent work of Marzuola and Simpson~\cite{MS} which is partially numerical (in the
spirit of Fibich, Merle, Raphael~\cite{FMR}). 
The construction of the manifold relies on a novel dispersive estimate due to Beceanu~\cite{Bec1} which allows for {\em small but not
decaying} and purely time-dependent lower-order perturbations to a Schr\"odinger operator. We rederive what is needed from~\cite{Bec1}  in our setting in Section~\ref{sec:AppendixB}. Section~\ref{sec:proofmain*} presents the proofs of Theorems~\ref{thm:main*} and~\ref{thm:threshold}, and they require the center-stable 
manifold of Section~\ref{sec:mf}. Section~\ref{sec:AppendixA} recalls some basic results related to the scattering theory of~\eqref{eq:Schr3} such as 
the Bahouri-Gerard decomposition in this setting, and the perturbation lemma needed for the Kenig-Merle method, and Section~\ref{appendix radSob} gives a proof for some radial Sobolev-type inequalities. 

The research in this paper as well as that of~\cite{NakS} is part of the wider area encompassing dispersive equations and their global existence theory
on the one hand, and the theory of unstable equilibria such as the ground state soliton on the other hand. Especially for the $L^2$ critical NLS equation 
substantial progress has been made on the very delicate blowup phenomena exhibited at and near the ground state. The $L^2$ critical equation is special
due to its invariance under the {\em pseudo-conformal transformation}, see for example~\cite{Caz}. Applying this class of transformations to the ground state~$Q$ 
gives rise to a solution blowing up in finite time, and it is unique with this property at exactly the mass of~$Q$, see Merle~\cite{Merle}. 
Bourgain, Wang~\cite{BourWang} studied the {\em conditional stability} of the pseudo-conformal blowup on a submanifold of large codimension,
and Krieger and the second author~\cite{KrS NLS} established the existence of a codimension~$1$ submanifold (albeit with no regularity and in a strong topology)
for which these solutions are preserved.  The conjecture that the pseudo-conformal should be stable under a codimension~$1$ condition is due to Galina Perleman~\cite{P}. 

A sweeping analysis of the {\em stable} blowup regime near the ground state for the $L^2$-critical case was carried out by Merle, Rapha\"el~\cite{MerRaph} in a series
of works, preceded by~\cite{P} which established the existence of the so-called $\log\log$ blowup regime. Very recently Merle, Rapha\"el, and Szeftel~\cite{MeRaSz2} proved that the Bourgain-Wang solutions are on the threshold between the $\log\log$ blowup and the scattering regimes. 
In~\cite{MeRaSz1} Merle, Rapha\"el and Szeftel were able to transfer some of the techniques from the critical case to the slightly $L^2$-supercritical one and
established stable blowup dynamics near the ground state in that case.  

The $L^2$-critical instability of the ground state is {\em algebraic} in nature rather than exponential, and thus very far from the considerations in this paper.
We emphasize that the hyperbolic dynamics is exploited strongly in our arguments. In addition, we rely heavily on the {\em radial} assumption, for example in the 
virial argument.

\section{The ground state and the  linearized operator}\label{sec:ground}

In this section we recall some variational and spectral properties around the ground states. 
The scaled family of ground states $Q(\al)=Q_\al:=\al Q(\al x)$ solves 
\EQ{
 \pt-\De Q_\al + \al^2 Q_\al = Q_\al^3,
 \pr\|\na Q_\al\|_2^2=\al\|\na Q\|_2^2, \pq \|Q_\al\|_4^4=\al\|Q\|_4^4, \pq \|Q_\al\|_2^2=\al^{-1}\|Q\|_2^2.}
Differentiating in $\al$ yields 
\EQ{
 (-\De + \al^2 - 3 Q_\al^2)Q_\al'=-2\al Q_\al.}
The relevant functionals in this paper are defined as
\EQ{\label{eq:EJK}
\pt E(u) = \| \nabla u\|_{2}^2/2-\| u\|_{4}^{4}/4, \pq M(u)=\|u\|_2^2/2, 
 \pr J(u)=\|\na u\|_2^2/2+\|u\|_2^2/2-\|u\|_4^4/4, 
 \pr K(u)=\|\na u\|_2^2-\frac{3}{4}\|u\|_4^4,
}
the first three being the conserved energy, mass, and action, respectively. 
The functional $K$ results from pairing $J'(u)$ with $(x\na+\na x)u/2$, the generator of dilations. 
 By construction, $Q$ is a critical point of $J$, i.e., $J'(Q)=0$
whence also $K(Q)=0$.  Moreover, the region 
\EQ{\label{M E 2}
M(u) E(u) < M(Q) E(Q)
}
is divided into two connected components by the conditions $\{K\ge0\}$ and $\{K<0\}$. 
The quantity $ME$ in~\eqref{M E 2} is scaling invariant and was used by Holmer, Roudenko~\cite{HolR} 
in their scattering analysis. 
The aforementioned division into two connected components is intimately linked to 
 the following minimization property. Define positive functional $G$ and $I$ by 
\EQ{ \label{def G}
 \pt G(\fy) :=J(\fy)-\frac{K(\fy)}{3}=\frac 16\|\na\fy\|_{L^2}^2+\frac 12\|\fy\|_{L^2}^2, 
 \pr I(\fy) :=J(\fy)-\frac{K(\fy)}{2}=\frac 12\|\fy\|_{L^2}^2+\frac 18\|\fy\|_{L^4}^4.}
\begin{lem} \label{minimization}
We have 
\EQ{ \label{cnstr min}
J(Q) \pt= \inf\{J(\fy) \mid 0\not=\fy\in H^1,\ K(\fy)=0\}   
 \pr= \inf\{G(\fy) \mid 0\not=\fy\in H^1,\ K(\fy)\le 0\}
 \pr= \inf\{I(\fy) \mid 0\not=\fy\in H^1,\ K(\fy)\le 0\},
}
and these infima are achieved only by $e^{i\theta}Q(x-c)$, with $\theta\in\R$ and $c\in\R^3$. 
\end{lem}
For the proof, see for example \cite[Lemma 2.1]{NakS} and \cite[Lemma 2.3]{IMN}. In particular, $J(\fy)<J(Q)$ implies either $\fy=0$, $K(\fy)>0$ or $K(\fy)<0$. 

Next, consider a decomposition of the solution in the form 
\EQ{
 u = e^{i\th}(Q+w).}
Inserting this into NLS yields
\EQ{ \label{eq w}
 i\dot w\pt=e^{-i\th}(\De u+|u|^2u+\dot\th u)
 \pr=(\De+\dot\th)(Q+w)+|Q|^2Q+2|Q|^2w+Q^2\bar{w}+2Q|w|^2+w^2\bar{Q}+|w|^2w
 \pr=(1+\dot\th)(Q+w)-\LL w+N(w),}
The $\R$-linear operator $\LL$ defined by\footnote{We need not extend $\LL$ as a $\C$-linear operator until Section 7, where we introduce a different notation. Hence the linear algebra for $\LL$ is always carried out in the sense of an $\R$-vector space.}
\EQ{\label{eq:RlinLdef}
 \LL w :=-\De w + w - 2 Q^2w - Q^2\bar{w},}
is self-adjoint on $L^2(\R^3;\C)$ with the inner product 
\EQ{
 \LR{f|g}:=\Re\int_{\R^3}f(x)\bar{g(x)}dx,}
and $N(w)$ is the nonlinear part defined by 
\EQ{
 N(w)=2Q|w|^2+Q w^2+|w|^2 w.}
Note that $i\LL$ is symmetric with respect to the symplectic form 
\[
 \Omega(f,g):=\Im \int_{\R^3}\bar{f(x)} g(x)dx=\LR{if|g}
\]
i.e., $\Om(i\LL f,g)=\Om(i\LL g,f)$. 
The generalized eigenfunctions of $i\LL$ are as follows: 
\EQ{\label{eq:GZpm}
 i \LL iQ=0, \pq i\LL Q'= -2iQ, \pq i\LL \GZ_\pm=\pm\mu \GZ_\pm,}
where $\mu>0$, 
\EQ{\label{eq:Q'def}
 Q'=\p_\al Q_\al|_{\al=1}=(1+r\p_r)Q, \pq \GZ_\pm=\fy\mp i\psi,} 
and with $\fy,\psi$  real-valued. In terms of the real and imaginary values, these equations are 
\EQ{
 \pt L_-Q=0, \pq L_+Q'=-2Q, 
 \pq L_-\psi = \mu\fy, \pq L_+\fy = -\mu\psi}
 with 
 \EQ{
 L_{-}=-\Delta + 1 - Q^{2},\qquad L_{+} = -\Delta + 1 - 3Q^{2}
 }
The existence of $\fy,\psi$ is standard and  follows from the minimization 
\EQ{
 \min\{\LR{\sqrt{L_-} L_+ \sqrt{L_-} f|f} \mid \|f\|_2^2\le 1\}<0}
 Recall that $L_-\ge0$ and $\ker(L_-)=\{Q\}$. In other words, $\LR{ L_- f|f}\gtrsim \|f\|_{H^1}^2$ if $f\perp Q$.  
After appropriate normalization of $(\fy,\psi)$, we have 
\EQ{
 \pt \LR{iiQ|Q'}=-\LR{Q|Q'}=M(Q),  \pq \LR{i\GZ_+|\GZ_-}=2\LR{\fy|\psi}=2\LR{L_-\psi|\psi}/\mu=1, 
 \pr 0=\LR{Q|\GZ_\pm}=\LR{iQ'|\GZ_\pm}=\LR{\fy|Q}=\LR{\psi|Q'}.}
Moreover $\LR{\psi|Q}\not=0$ and so we can choose $\LR{\psi|Q}>0$. To see this, suppose $\psi\perp Q$, then $\fy\perp L_+Q=-2Q^3$, and so 
by Lemma~2.3 of~\cite{NakS},  
$0\le\LR{L_+\fy|\fy}=-\mu\LR{\psi|\fy}<0$, which is a  contradiction. 

The symplectic decomposition of $L^2(\R^3;\C)$ corresponding to these discrete modes 
 is uniquely given by 
\EQ{
 \pt f=ai Q+bQ'+c_+\GZ_++c_-\GZ_- + \eta,
 \pr a=\LR{if|Q'}/M(Q), \pq b=-\LR{f|Q}/M(Q), \pq c_\pm=\pm\LR{if|\GZ_\mp},}
One has  $0=\LR{\eta |Q}=\LR{i\eta |Q'}=\LR{i\eta|\GZ_\pm}$ and 
the symplectic projections onto $\{iQ,Q'\}^{i\perp}$ and onto $\{\GZ_\pm\}^{i\perp}$ commute. 

We apply the symplectic decomposition to $w$. Then, writing $\gamma= ai Q+bQ'+\eta$ one has 
\EQ{ \label{u decop}
 u=e^{i\th}(Q+w)=e^{i\th}(Q+\la_+\GZ_+ + \la_-\GZ_-+\ga)} 
 The justification for including the ``root''-part (i.e., the zero modes) in $\gamma$ follows from
  a suitable choice of the symmetry  parameters $\alpha,\theta$, see Section~\ref{sec:mod}. 
 The action is expanded as 
\EQ{\label{eq:J_exp}
 J(u)-J(Q)\pt=\frac 12\LR{\LL w|w} - C(w)
 \pn=-\mu\la_+\la_-+\frac 12\LR{\LL\ga|\ga}-C(w),}
where the superquadratic part $C(w)$ is defined by 
\EQ{ \label{def C}
 C(w)=\LR{|w|^2w|Q}+\|w\|_4^4/4.}
 The following lemma will guarantee the 
 positivity of the $\ga$ component in \eqref{eq:J_exp}.
 
\begin{lem} \label{L pos}
Let $f,g\not=0$ be real-valued,  radial and  satisfy 
\EQ{
 \LR{f|\psi}=0=\LR{g|Q'}.}
Then $\LR{L_+f|f}\simeq\|f\|_{H^1}^2$ and $\LR{L_-g|g}\simeq\|g\|_{H^1}^2$.
\end{lem}
\begin{proof}
Let $f\not=0$ satisfy $f\perp\psi$ and $\LR{L_+f|f}\le 0$. Then 
\EQ{
 \LR{L_+f|\fy}=\LR{f|-\mu\psi}=0, \pq \LR{L_+\fy|\fy}=-\mu\LR{\psi|\fy}=-\mu/2<0,}
so $f$ and $\fy$ are not colinear, and moreover 
\EQ{
 \LR{L_+(af+b\fy)|af+b\fy}=a^2\LR{L_+f|f}+b^2\LR{L_+\fy|\fy}+2ab\LR{L_+f|\fy}
 \le 0}
for any $a,b\in\R$, which contradicts the fact that $L_+$ has only one nonpositive eigenvalue (cf.~for example, Lemma~2.3 in~\cite{NakS}). 

Next, apply the orthogonal projection of $Q$ to $g$: 
\EQ{
 g=cQ+g', \pq g'\perp Q,}
then $\LR{g|Q'}=0$ implies that $c=-\LR{g'|Q'}/\LR{Q|Q'}$. Hence 
\EQ{
 \|g\|_{H^1}^2\simeq c^2+\LR{L_-g'|g'}\lec \|g'\|_{H^1}^2 \simeq \LR{L_-g|g}}
 as desired. 
\end{proof}

The spectrum of $L_-$ in $L^2_{\mathrm{rad}}$ consists of $0$ as a ground state (simple eigenvalue), $[1,\I)$ as essential
spectrum (which is absolutely continuous); $L_+$ (again over the radial subspace) has a ground state with eigenvalue $-k^2<0$, no other
eigenvalues in $(-k^2, \eps)$ where $\eps>0$, and the same essential spectrum as $L_-$. These properties are well-known and
easy to obtain via variational arguments, see for example~\cite[Lemma 2.1]{NakS}.  More delicate is the question of eigenvalues in the gap $(0,1]$
and what the behavior is at the threshold~$1$.  This question turns out to be irrelevant for the proof of Theorem~\ref{thm:main}, but is relevant once
the center-stable manifold comes into play, at least with the approach that is implemented here (Lyapunov-Perron method).  
For the cubic nonlinearity, as it is being considered here, \cite{DS} gives numerical evidence that $L_\pm$ have no eigenvalues in $(0,1]$
and that $1$ is a regular threshold (no resonance there).

\section{Parameter choice}\label{sec:mod}

In this section, we determine the modulation parameters of the ground state part, so that we can translate the local arguments from the Klein-Gordon case \cite{NakS} (where the ground state is fixed) to the modulation analysis for NLS. In particular, we will derive the ejection lemma and the variational lower bounds in the same spirit as in \cite[Lemmas 4.2 and 4.3]{NakS}. 

We determine $\al,\th$ explicitly by the equations 
\EQ{\label{eq:para}
 M(u)=M(Q_\al), \pq (u|e^{i\th}Q_\al' )<0,}
where $(f|g)=\int f(x)\bar{g(x)}dx$. 
Indeed, both formulae can be explicitly solved by 
\EQ{
 \al=M(Q)/M(u), \pq \th=\Im\log(u|-Q_\al').}
Since $(Q_\al|Q_\al')=-\al^{-2}M(Q)<0$, there is a unique solution $(\al,e^{i\th})\in(0,\I)\times S^1$ as long as $u$ is close to some $e^{i\th}Q_\al$. 
It is easy to see that $u=e^{i\fy}Q_\be$ gives $\fy=\th$ and $\al=\be$.  Even though this choice of parameters differs from the 
traditional one used in ``modulation theory'' (see Section~\ref{sec:mf} for the latter)  we find that~\eqref{eq:para}  
is convenient for our purposes.  Loosely speaking, up until Section~\ref{sec:mf} we will be working  more in the spirit of orbital
stability theory, whereas Section~\ref{sec:mf} requires the finer asymptotic stability property and thus a different handling of the modulation
parameters. 

The advantage of this choice of $(\al,\theta)$  is that it is explicit and moreover $M(u)$ is conserved in time, and so $\al$ is fixed. 
A disadvantage is that it is nonlinear, in the sense that 
\EQ{
 \LR{w|Q_\al}=-M(w), \pq \LR{iw|Q'_\al}=0,}
but this will be a higher order effect that can be ignored (we assume throughout this section that $w$ is small). 
Without loss of generality we now fix 
\EQ{
 \al=1, \pq M(u)=M(Q)}
and omit $\al$. We can further decompose 
\EQ{\label{eq:lam+-}
 u=e^{i\th}(Q+w), \pq w=\la_+\GZ_++\la_-\GZ_-+\ga,
 \pq \la_\pm=\pm\LR{iw|\GZ_\mp}.}
Moreover, define
\EQ{
 \la_1:=(\la_++\la_-)/2, \pq \la_2:=(\la_+-\la_-)/2, \pq \vec\la:=(\la_1,\la_2).}
so that the decomposition is written as 
\EQ{
 w=2\la_1\fy-2i\la_2\psi+\ga.}
The remainder's orthogonality is given by
\EQ{
 \LR{\ga|Q}=-\frac12\|w\|_2^2, \pq \LR{i\ga|Q'}=0, \pq \LR{i\ga|\GZ_\pm}=0.}
which, by Lemma \ref{L pos}, is sufficient for the property 
\EQ{
 \LR{\LL\ga|\ga}\simeq\|\ga\|_{H^1}^2.}
The equation of $\th$ is obtained by differentiating 
$0=\LR{iu|e^{i\th}Q'}=\LR{iw|Q'}$. Using the equation of $w$ \eqref{eq w}, as well as $\LR{w+2Q|w}=0$, one concludes that 
\EQ{
 (\dot\th+1)[M(Q)-\LR{w|Q'}]=\LR{-\LL w+N(w)|Q'}
 = -\|w\|_2^2+\LR{N(w)|Q'}.}
The equation for $\la_\pm$ is obtained by differentiating \eqref{eq:lam+-}. In fact, 
\EQ{
 \dot\la_\pm \pt= \pm\LR{i\dot w|\GZ_\mp} = \LR{(\dot\th+1)(Q+w)-\LL w+N(w)|\pm\GZ_\mp}
 \pr= \pm\mu\la_\pm + N_\pm(w), \\
  N_\pm(w) &:=\LR{N(w)+(\dot\th+1)w|\pm\GZ_\mp},}
and so $\vec\la$ solves 
\EQ{
 \pt \dot\la_1 = \mu\la_2 + N_1(w), \pq N_1(w)=\LR{N(w)+(\dot\th+1)w|i\psi},
 \pr \dot\la_2 = \mu\la_1 + N_2(w), \pq N_2(w)=\LR{N(w)+(\dot\th+1)w|\fy}.}
Recall the energy expansion
\EQ{ \label{ene exp}
 J(u)-J(Q)\pt=-\mu\la_+\la_-+\frac12\LR{\LL\ga|\ga}-C(w)
 \pr=\mu[\la_2^2-\la_1^2]+\frac12\LR{\LL\ga|\ga}-C(w).}
We therefore define the linearized energy norm to be
\EQ{\label{eq:linEnorm}
 \|v\|_{E}^2\pt:=\mu|\vec\la|^2+\frac12\LR{\LL\ga|\ga}
 \pn=\frac{\mu}{2}(\la_+^2+\la_-^2)+\frac12\LR{\LL\ga|\ga} \simeq \|v\|_{H^1}^2,}
 where we used Lemma~\ref{L pos} for the final step. 
Furthermore, we define the smooth nonlinear distance function in such a way that, 
still under the mass constraint $M(u)=M(Q)$, 
\EQ{\label{eq:dQdef}
\pt d_{Q} ^{2}  (u)\simeq  \inf_{\be\in\R} \| u-e^{i\be} Q\|_{H^{1}}^{2} 
\pr d_Q^2(u) = \|w\|_{E}^2-\chi(\|w\|_{E}/(2\de_E))C(w)\text{\ if\ } d_{Q}(u)\ll 1,}
where $\de_E\ll 1$ is chosen such that 
\EQ{
 \|v\|_{E} \le 4\de_E \implies |C(v)|\le \|v\|_{E}^2/2.} 
 The smooth cut-off $\chi(r)$ is equal to one on $|r|\le 1$ and vanishes for~$|r|\ge2$. 
To see the consistency of the above two properties, let $u=e^{i\be}Q+v$ be a minimizer for $$\dist_{H^1}(u,\cS_1)=\inf_\be\|u-e^{i\be}Q\|_{H^1}.$$ 
Then $\LR{w|iQ'}=0$ implies that $\LR{e^{-i\th}v|iQ'}=\sin(\be-\th)M(Q)$, and so $$\|v\|_{H^{-1}}\gec \inf_{k\in\Z}|\be-\th+k\pi|$$ as long as $v$ is small.
The case of $k$ odd can be eliminated here via the sign in~\eqref{eq:para}. Indeed,  by the second condition in~\eqref{eq:para}, 
\EQ{
-\cos(\beta-\theta) (Q|Q') > \Re(v|e^{i\theta} Q')
}
which 
excludes that $\beta-\theta$ lies near an odd multiple of~$\pi$. 
Therefore, 
\EQ{
 \|w\|_E \simeq \|w\|_{H^1} \lec \|v\|_{H^1} \le \|w\|_{H^1}.} 
By the same argument, if $u=e^{i\be}Q+v$ is an $L^2$ distance minimizer, then 
\EQ{ \label{param low ord}
 |\be-\th|\lec \dist_{L^2}(u,\cS_1),}
provided that the right-hand side is small enough. 

In the region $d_Q(u)\ll 1$, the distance function $d_Q$ enjoys the following properties: 
\EQ{ \label{energy dist}
 \pt\|w\|_{E}^{2}/2\le d^{2}_Q(u) \le 2\|w\|^{2}_{E}, \pq d_Q^2(u)=\|w\|_{E}^2+O(\|w\|_{E}^3), 
 \pr d_Q(u)\le\de_E \implies d^{2}_Q(u)=J(u)-J(Q)+2\mu\la_1^2.}
Hence as long as $d_Q(u)<\de_E$ we have 
\EQ{\label{eq:dQu_diff}
 \p_td_Q^2(u)\pt=4\mu\la_1\dot\la_1
 \pn=4\mu^2\la_1\la_2+4\mu\la_1N_1(w).}

\begin{lem} \label{lem:eigen}
For any $u\in\HH_1$ satisfying  
\EQ{
 J(u)<J(Q)+d_Q(u)^2/2, \pq d_Q( u)\le \de_E,}
one has $d_Q(u)\simeq|\la_1|=-\sg\la_{1}$ for $\sg=\pm 1$. 
\end{lem}
\begin{proof}
\eqref{energy dist} yields
\EQ{ \label{eq:dQJJ} 
 d_Q^2(u)=J(u)-J(Q)+2\mu\la_1^2<d_Q^2(u)/2+2\mu\la_1^2.}
and so, $\mu\la_1^2/2\le\|w\|_E^2/2 \le d_Q^2(u)<4\mu\la_1^2$. The second inequality sign
uses~\eqref{energy dist}, whereas the final one uses~\eqref{eq:dQJJ}. 
\end{proof}

It will be convenient to relate $d_Q(u)$ to the $L^2$ distance, taking advantage of the $H^1$ subcriticality of our nonlinearity. 

\begin{lem} \label{lem:dist L2}
Let $u\in\HH_1$ satisfy 
\EQ{
 \|u\|_{H^1}\lec 1, \pq J(u)-J(Q)\ll\de_E^2, \pq J(u)-J(Q)<d_Q(u)^2/2.}
Then we have 
\EQ{
 \dist_{L^2}(u,\cS_1)=\inf_{\be\in\R}\|u-e^{i\be}Q\|_2 \gec\min(d_Q(u),\de_E^2).}
\end{lem}
\begin{proof}
Consider the decomposition \eqref{u decop} such that $\|w\|_2\simeq \dist_{L^2}(u,\cS_1)$. This is legitimate by \eqref{param low ord}.
We may assume $\dist_{L^2}(u,\cS_1)\ll\de_E^2$. 
Then using Gagliardo-Nirenberg, we obtain 
\EQ{ 
 |C(w)|\lec\|w\|_2\ll\de_E^2, \pq |\la_\pm| \lec \|w\|_2\ll\de_E^2,}
and so from \eqref{eq:J_exp}, 
\EQ{
 \|\ga\|_{H^1_x}^2 \simeq \frac12\LR{\LL\ga|\ga}=J(u)-J(Q)+\mu\la_+\la_-+C(w) \ll \de_E^2,}
hence $d_Q(u)\lec\|w\|_{H^1_x}\lec|\la_+|+|\la_-|+\|\ga\|_{H^1_x}\ll\de_E$. 
Then the previous lemma implies that $d_Q(u)\sim|\la_1|$, whence $d_Q(u)\lec \dist_{L^2}(u,\cS_1)$ as desired. 
\end{proof}

The following lemma exhibits the mechanism by which solutions are ejected along the unstable mode. 

\begin{lem} \label{lem:eject}
There exists a constant $0<\de_X\le \de_E$, as well as constants $C_*,T_*\simeq 1$ with the following properties: Let $u(t)$ be a local solution of NLS in $\HH_1$ on an interval $[0,T]$ satisfying 
\EQ{
 R:=d_Q(u(0)) \le \de_X, \pq J(u)<J(Q)+R^2/2} 
and for some $t_0\in(0,T)$, 
\EQ{ \label{exiting condition}
 d_Q(u(t)) \ge R \pq (0<\forall t<t_0).}
Then $u$ extends as long as $d_Q(u(t))\le\de_X$, and satisfies $\forall\; t\ge0$ 
\EQ{\label{eq:eject_est}
  \pt d_Q(u(t)) \simeq -\sg\la_1(t) \simeq -\sg\la_+(t) \simeq e^{\mu t}R, 
  \pr |\la_-(t)|+\|\ga(t)\|_E \lec R+e^{2\mu t}R^2, 
  \pr \sg K(u(t)) \gec (e^{\mu t} - C_*)R,}
where $\sg=+1$ or $\sg=-1$ is constant. Moreover, $d_Q(u(t))$ is increasing for $t\ge T_*R$, and $|d_Q(u(t))-R|\lec R^3$ for $0\le t\le T_*R$. 
\end{lem}
\begin{proof}
Lemma \ref{lem:eigen} yields $d_Q(u)\simeq-\sg\la_1$ with $\sg=\pm 1$ fixed, as long as $R\le d_Q(u)\le \de_E$.
The exiting condition \eqref{exiting condition} implies $\p_td_Q(u)^2|_{t=0}\ge 0$.
Since $|N_1(w)|\lec\|w\|_{H^1}^2\lec \la_1^2$, we deduce from~\eqref{eq:dQu_diff} that $-\sg\la_2(0)\gec-|\la_1(0)|^2$ and so $\la_+(0)\simeq\la_1(0)$. 
 
  Integrating the equation for $\la_\pm$ yields
\EQ{
 |\la_\pm(t)-e^{\pm\mu t}\la_\pm(0)| \lec \int_0^t e^{\mu(t-s)}|N_\pm(w(s))|ds \lec \int_0^t e^{\mu(t-s)}|\la_1(s)|^2ds,}
from which by continuity in time we deduce that as long as $Re^{\mu t}\ll 1$, 
\EQ{\label{eq:la_1}
 \la_1(t) \simeq \la_+(t) \simeq -\sg Re^{\mu t}, \pq |\la_\pm(t)-e^{\pm\mu t}\la_\pm(0)| \lec R^2e^{2\mu t}.}
 
   Now consider the nonlinear energy projected onto the $\GZ_\pm$ plane: 
\EQ{
 E_\GZ(\la):=-\mu\la_+\la_--C(\la_+\GZ_++\la_-\GZ_-),}
where $C(\cdot)$ is defined in \eqref{def C}. 
Using the equation of $\la_\pm$, we obtain
\EQ{
 \p_tE_\GZ\pt=-\mu\la_+\dot\la_--\mu\la_-\dot\la_+-\LR{N(\la_+\GZ_++\la_-\GZ_-)|\dot\la_+\GZ_++\dot\la_-\GZ_-}
 \pr =  \LR{ N(w) - N(\la_+\GZ_++\la_-\GZ_-)|\dot\la_+\GZ_++\dot\la_-\GZ_-}   
 \pr \qquad +   (\dot\theta+1) \LR{ w  | \dot\la_+\GZ_++\dot\la_-\GZ_- } 
 \pr\lec \|\ga\|_{H^1}^2|\la|+|\la|^4.}
Hence $|\p_t(J(u)-E_\GZ)|\lec \|\ga\|_{H^1}^2|\la|+|\la|^4$, while 
\EQ{
 J(u)-J(Q)-E_\GZ \pt=\LR{\LL\ga|\ga}/2-C(w)+C(\la_+\GZ_++\la_-\GZ_-)
 \pr\simeq\|\ga\|_{H^1}^2+O(\|\ga\|_{H^1}|\la|^2),}
and so
\EQ{ \label{eng-bd ga}
 \|\ga\|_{L^\I_tH^1(0,T)}^2 \lec \|\ga(0)\|_{H^1}^2+\|\ga\|_{L^\I_tH^1(0,T)}\|\la\|_{L^\infty(0,T)}^2+\|\la\|_{L^4(0,T)}^4,}
which is sufficient. 
Indeed, the desired estimate on $\ga$ follows from~\eqref{eq:la_1} inserted into~\eqref{eng-bd ga}. 
The equation of $\la_2$ implies $-\sg\la_2(t)\gec R(e^{\mu t}-1)-O(R^2)$, hence there is $T_*\simeq 1$ 
such that $-\sg\la_2\gec R$ and $\p_td_Q(u)>0$ for $t\ge T_*R$. 
For $0\le t\le T_*R$, we have $|\p_td_Q(u)|\lec R^2$ and so $|d_Q(u)-R|\lec R^3$. 

 Finally, we expand $K$ around $Q$:
\EQ{ \label{exp K2}
 K(u)\pt=K(Q+w)=\LR{-2Q+L_+Q/2|w_1}+O(\|w\|_{H^1}^2)
 \pr=-\la_1\mu\LR{Q|\psi}-\LR{2Q+Q^3|\ga}+O(\|w\|_{H^1}^2).}
Since $\LR{Q|\psi}>0$, we obtain the desired bound on $K$ from the behavior of $\la_1$. 
\end{proof}

The following lemma gives lower bounds on $|K|$, which should be used once the solution is away from $\cS$, or after 
being ejected from a neighborhood thereof, as described by Lemma~\ref{lem:eject}. For the definition of the functional~$I$ see~\eqref{def G}. 

\begin{lem} \label{lem:vari}
For any $\de>0$, there exist $\e_0(\de), \ka_0, \ka_1(\de),\ka_2(\de)>0$ such that the following hold: (I) For any $u\in\HH_1$ satisfying $J(u)< J(Q)+\e_0(\de)^2$, and $d_Q(u) \ge \de$, we have 
\EQ{ \label{-K bd}
 K(u) \le -\ka_1(\de),\pq \text{or} \pq  
 K(u) \ge \min(\ka_1(\de),\ka_0\|\na u\|_2^2).}
(II) For any $u\in\HH$ satisfying $I(u)<J(Q)-\de$, we have 
\EQ{
 K(u) \ge \min(\ka_2(\de),\ka_0\|\na u\|_2^2).}
\end{lem}
\begin{proof}
Part (I) is proved in the same way as  \cite[Lemma 4.3]{NakS}. In fact, the situation here is simpler because in contrast to~\cite{NakS} $d_Q$ does not contain the time derivative 
of the solution and we are dealing with only one $K$ functional. Part~(II) is proved via an analogous argument. First, since $M(u)\le I(u)$ is bounded, the Gagliardo-Nirenberg inequality 
\EQ{ \label{GN}
 \|u\|_4^4\lec\|\na u\|_2^3\|u\|_2}
implies that if $\|\na u\|_2\ll 1$ then $K(u)\simeq \|\na u\|_2^2$. Hence we may assume that $\|\na u\|_2\gec 1$. 
Suppose towards a contradiction that $u_n\in\HH$ satisfy $I(u_n)<J(Q)-\de$, $\|\na u_n\|_2\gec 1$ and $K(u_n)\to 0$. Since both $I(u_n)$ and $K(u_n)$ are bounded, the sequence $\{u_n\}$ is bounded in $H^1$. Hence by extraction of a subsequence, we may assume that $u_n\to  u_\I$ weakly in $H^1$ and strongly in $L^4$. Then $I(u_\I)\le J(Q)-\de$ and $K(u_\I)\le 0$, so Lemma~\ref{minimization} implies that $u_\I=0$, hence $\|u_n\|_4\to 0$, which contradicts that $\|\na u_n\|_2\gec 1$ and $K(u_n)\to 0$. 
\end{proof}

Combining the above lemmas, we can now define the sign function $\Sg$ which determines the fate of solutions passing by $\cS$. The proof is the same as for 
the analogous statements~\cite[Lemmas 4.4 and 4.5]{NakS}. 

\begin{lem} \label{lem:sign}
Let $\de_S:=\de_X/(2C_*)>0$ where $\de_X$ and $C_*\ge 1$ are the constants from Lemma \ref{lem:eject}. Let $0<\de\le\de_S$ and 
\EQ{ \label{def HS}
 \HH_{(\de)}:=\{u\in\HH_1 \mid J(u)<J(Q)+\min(d_Q(u)^2/2,\e_0(\de)^2)\},} 
where $\e_0(\de)$ is given by Lemma \ref{lem:vari}. Then there exists a unique continuous function $\Sg:\HH_{(\de)}\to\{\pm 1\}$ satisfying 
\EQ{ \label{def Sg}
 \CAS{u\in\HH_{(\de)},\ d_Q(u)\le\de_E &\implies  \Sg(u)=-\sign\la_1,\\ 
 u\in\HH_{(\de)},\ d_Q(u)\ge\de &\implies \Sg(u)=\sign K(u),}}
where we set $\sign 0=+1$. In addition, we have 
\EQ{
 \sup\{\|u\|_{H^1} \mid u\in\HH_{(\de_S)},\ \Sg(u)=+1\}\lec 1.}
\end{lem}

\section{Virial argument and the one-pass theorem}\label{sec:onepass}

In this section we establish the following one-pass theorem by means of a suitable virial argument.

\begin{thm}\label{thm:onepass}
There exist $0<\e_*\ll R_*\ll 1$ with the following property: let $u\in C([0,T);\HH)$ be a forward maximal solution of~\eqref{eq:Schr3} satisfying $M(u)=M(Q)$, $J(u)<J(Q)+\e^{2}$ and $d_{Q}(u(0))<R$ for some $\e\in(0,\e_*]$ and $R\in (2\e,R_*]$. 
Then one has the following dichotomy: either $T=\infty$ and $d_Q(u(t))<R+R^2$ for all $t\ge0$, 
or $d_Q(u(t))\ge R+R^2$ on $t_*\le t<T$ for some finite $t_{*}>0$. In the latter case, $\Sg(u(t))\in\{\pm 1\}$ does not change on $t_{*}\le t<T$; if it is $-1$, then $T<\I$, whereas if it is $+1$, then $T=\I$. 
\end{thm}
In the Klein-Gordon case, we were able to use the same $R$ for the dichotomy because the distance function was strictly convex in $t$. 
For NLS it may exhibit oscillations on the order of $O(R^3)$, and so we need some room (we chose $R^2$) to ensure a true ejection from the small neighborhood. 

In Section~\ref{sec:scat} we will show that the solution in fact scatters to zero if $\Sg(u(t))=+1$ for large time. The proof of Theorem~\ref{thm:onepass} will take up this entire section. 
In fact, most work goes into proving the no-return statement, as the finite time blowup vs.~global existence dichotomy then follows easily. 
Indeed, the global existence in the $\Sg=+1$ region readily follows from the a priori $H^1$ bound in Lemma \ref{lem:sign}. 

We now turn to the details. 
We may assume that $u$ does not stay very close to $\cS$ for all $t>0$, so that we can apply the ejection Lemma \ref{lem:eject} at some time $t_*>0$. 
Recall Ogawa-Tsutsumi's {\it saturated virial identity} \cite[(3.5)]{OgTs}
\EQ{\label{eq:OgTsu}
 \pn\p_t\LR{\phi_m u|iu_r}
 \pt= \int_{\R^3} 2|u_r|^2\p_r\phi_m \, dx
  - |u|^2\De(\p_r/2+1/r)\phi_m\, dx
 \prQ -\int_{\R^3}  |u|^4(\p_r/2+1/r)\phi_m\, dx,}
where the smooth bounded radial function $\phi_m$ is chosen as follows: 
\EQ{
 \phi_m(r)=m\phi(r/m), \pq \phi_m(0)=0\le \phi_m'(r)\le 1=\phi_m'(0), \pq \phi_m''(r)\le 0.}
Notice that with this choice of $\phi_{m}$, eq.~\eqref{eq:OgTsu}  is not merely a cut-off of the virial identity, 
but rather a ``smooth interpolate'' of the latter with the Morawetz estimate for large $|x|$. 
This is indeed crucial for the following arguments, which are slightly more delicate than those in~\cite{OgTs}. 

The idea (as in \cite{NakS}) is now to combine the hyperbolic structure of Lemma~\ref{lem:eject} close to $\cS$ with the variational structure in Lemma~\ref{lem:vari} away from $\cS$, in order to control the virial identity through $K(u)$. We choose $\delta_{*}>0$ as the distance threshold between the two regions in $\HH_1$: for $d_Q(u)<\de_*$ we use the hyperbolic estimate in 
Lemma~\ref{lem:eject}, and for $d_Q(u)>\de_*$ we use the variational estimate in Lemma~\ref{lem:vari}. Hence $\de_*,\e_*,R_*$ should satisfy 
\EQ{ \label{cond0 Rede*}
 \e_* \ll R_* \ll \de_* \ll \de_S, \pq \e_*\le \e_0(\de_*).}
Below, we shall impose further  smallness conditions on $\de_*,R_*,\e_*$. 
Afterward, $R_*$ and then $\e_*$ need to  be made even smaller in order to satisfy the above conditions, depending on $\de_*$. 

Suppose towards a contradiction that $u$ solves the NLS equation~\eqref{eq:Schr3} on $[0,T)$ in $\HH_1$ satisfying for some $0<T_1<T_2<T_3<T$ and all $t\in(T_1,T_3)$, 
\EQ{
 d_Q(u(0))< R = d_{Q}(u(T_{1}))=d_{Q}(u(T_{3})) < d_Q(u(t)), \pq d_Q(u(T_2))\ge R+R^2,}
as well as $J(u)<J(Q)+\e^2$, for some $\e\in(0,\e_*]$ and $R\in(2\e,R_*]$. 

Lemma \ref{lem:sign} implies that $\sg:=\Sg(u(t))\in\{\pm 1\}$ is well-defined and constant on $T_1\le t\le T_3$. 

We apply the ejection Lemma~\ref{lem:eject} first from $t=T_1$ forward in time. Then by the lemma, there exists $T_1'\in(T_1,T_1+T_*R)$ such that $d_Q(u(t))$ increases for $t>T_1'$ until it reaches $\de_X$, and $d_Q(u(T_1'))=R+O(R^3)<d_Q(u(T_2))\ll\de_X$. Hence $T_1<T_1'<T_2$, and by the lemma there is $T_1''\in(T_1',T_3)$ such that $d_Q(u(t))$ increases exponentially on $(T_1',T_1'')$, $d_Q(u(T_1''))=\de_X$ and on $(T_1,T_1'')$, 
\EQ{ \label{eject T1}
 d_Q(u(t))\simeq e^{\mu(t-T_1)}R, \pq \sg K(u(t))\gec (e^{\mu(t-T_1)}-C_*)R.}
We can argue in the same way from $t=T_3$ backward in time to obtain a time interval $(T_3'',T_3)\subset(T_1'',T_3)$, so that $d_Q(u(T_3''))=\de_X$,
\EQ{ \label{eject T3}
 d_Q(u(t))\simeq e^{\mu(T_3-t)}R, \pq \sg K(u(t))\gec (e^{\mu(T_3-t)}-C_*)R \pq(T_3''<t<T_3),}
and $d_Q(u(t))$ is decreasing in the region $d_Q(u(t))\ge R+R^2$. 
Moreover, from any $\tau\in (T_1'',T_3'')$ where $d_Q(u(\tau))<\de_*$ is a local minimum, we can apply the ejection lemma both forward and backward in time, thereby obtaining an open interval $I_\tau\subset(T_1'',T_3'')$ so that $d_Q(u(\p I_\tau))=\{\de_X\}$, 
\EQ{ \label{eject tau}
 d_Q(u(t))\simeq e^{\mu|t-\tau|}d_Q(u(\tau)), \pq \sg K(u(t))\gec (e^{\mu|t-\tau|}-C_*)d_Q(u(\tau)) \pq(t\in I_\tau),}
and $d_Q(u(t))$ is monotone in the region $d_Q(u(t))\ge 2d_Q(u(\tau))$, which is the reason for $I_\tau\subset(T_1'',T_3'')$. Moreover, the monotonicity away from $\tau$ implies that any two  intervals $I_{\tau_1}$ and $I_{\tau_2}$ for distinct minimal points $\tau_1$ and $\tau_2$ are either disjoint or identical.  
Therefore, we have obtained disjoint open subintervals $I_1,\dots,I_n\subset (T_1,T_3)$ with $n\ge 2$, where we have either \eqref{eject T1}, \eqref{eject T3}, or \eqref{eject tau} with $\tau=\tau_j\in I_j$, and at the remaining times 
\EQ{
 t\in I':=(T_1,T_3)\setminus\Cu_{j=1}^n I_j,}
we have $d_Q(u(t))\ge\de_*$, so that we can apply Lemma \ref{lem:vari} to obtain 
\EQ{ \label{vari I'}
 \CAS{ K(u(t))\ge \min(\ka_1(\de_*),\ka_2\|\na u(t)\|_2^2) &(\sg=+1),\\
 K(u(t))\le-\ka_1(\de_*) &(\sg=-1).} \pq (t\in I')} 

\subsection{Virial estimate in the blow-up case $\sg=-1$}
In this case, we choose $\phi$ just as in \cite{OgTs}: 
\EQ{
 \phi(r) = \CAS{r &(r\le 1), \\ \frac{3}{2} &(r\ge 2),}}
and then rewrite \eqref{eq:OgTsu} in the form 
\EQ{
 \p_t\LR{\phi_m u|iu_r} = 2K(u) - 2\int |u_r|^2 f_{0,m}\, dx + \int [|u|^2f_{1,m}/r^2 + |u|^4f_{2,m}] \,  dx,}
where $f_{j,m}=f_j(r/m)$ are smooth functions supported on $r>m$, defined by 
\EQ{
 f_0=1-\phi_r, \pq f_1=-r^2\De(\p_r/2+1/r)\phi, \pq f_2=3/2-(\p_r/2+1/r)\phi.}
For the $L^4$ error term we use the radial Sobolev inequality as in \cite{OgTs}
\EQ{
 \|u\|_{L^4(r>m)}^4 \lec m^{-2}\|u\|_{L^2(r>m)}^3\|u_r\|_{L^2(r>m)},}
see \eqref{rad Sob sg-}. In order to absorb the kinetic term, noting that $f_0'\ge 0$ and $|f_2|\lec f_0$, we use the weighted version of the above inequality: 
\EQ{ \label{weight Sobolev}
 \pt\int_m^\I f_{0,m}(r)|u|^4(r) r^{2} \, dr
 =\int_m^\I\int_s^\I f_{0,m}'(s)|u|^4(r) r^{2}  \, dr ds
 \pr\lec\int_m^\I f_{0,m}'(s)s^{-2}\|u\|_{L^2(r>s)}^3\|u_r\|_{L^2(r>s)} \, ds
 \pr\lec\int_m^\I f_{0,m}'(s)\int_s^\I \la|u_r|^2(r) r^{2}\, drds
 \pn+\int_m^\I f_{0,m}'(s) \la^{-1} s^{-4}\| u \|_{L^2(r>s)}^6 \, ds
 \pr\le \la\int_m^\I f_{0,m}(r)| u_r|^2(r) r^{2} \, dr 
  + \la^{-1}m^{-4}\| u \|_{L^2(r>m)}^6,}
which holds uniformly for $\la>0$. Choosing $\la>0$ small (in terms of the constants in those Sobolev inequalities), we obtain 
\EQ{\label{eq:Knegmonotone}
 \p_t\LR{\phi_m u |iu_r}
& \le 2K(u) + O(m^{-4}\|u\|_{L^2(r>m)}^6) + O(m^{-2}\|u\|_{L^2(r>m)}^2)\\
&\le 2K(u) + O(m^{-2}).} 
 We can now prove Theorem~\ref{thm:onepass} in the blow-up case $\sg=-1$, by integrating \eqref{eq:Knegmonotone}, combined with \eqref{eject T1}--\eqref{vari I'}. We thus
  obtain 
\EQ{ \label{virial decrease}
 -[\LR{\phi_m u |iu_r}]_{T_1}^{T_3} \pt\gec \sum_{j=1}^n\int_{I_j}[(e^{\mu|t-\tau_j|}-C_*)d_Q(u(\tau_j))-Cm^{-2}]\, dt 
 \prQ+ \int_{I'}[\ka_1(\de_*)-Cm^{-2}]\, dt
 \pr\gec n\de_X \ge \de_X,}
provided that 
\EQ{
 m^{-2} \lec R_*, \pq m^{-2} \ll \ka_1(\de_*).}
On the other hand, since $d_Q(u(t))=R$ at $t=T_1,T_3$ and since $Q$ is exponentially decaying, 
\EQ{\label{eq:changemuch}
 \Bigl|[\LR{\phi_m u |iu_r}]^{T_3}_{{T_{1}}}\Bigr|  \lec R + m R^{2} \lec R_* \ll \de_X, 
}
if we choose $m=1/R$. 
Comparing this bound with~\eqref{virial decrease} leads to a contradiction. 
In conclusion, the solution $u(t)$ cannot return to the $R$-ball from the $\sg=-1$ side if we choose $\de_*,R_*,\e_*>0$ such that 
\EQ{
 R_*^2 \ll \ka_1(\de_*)}
and \eqref{cond0 Rede*} are satisfied. 
Therefore, if $u$ extends to $t\to+\I$, then $T_3=\I$ and so \eqref{virial decrease} with $m=1/R$ fixed implies 
\EQ{
 m\|u_r(t)\|_2 \gec -\LR{\phi_mu|iu_r} \to \I}
as $t\to\I$. Hence for large $t\gg 1$ we have $K(u(t))=3E(u(t))-\frac12\|\nabla u(t)\|_{2}^{2} \to-\I$. 
Thus for $T_1\ll\forall t_1<\forall t_2$, 
\EQ{
 [\LR{\phi_mu|iu_r}]_{t_1}^{t_2} \lec -\int_{t_1}^{t_2}\|u_r(t)\|_2^2\, dt,}
and so 
\EQ{
 m\|u_r(t_2)\|_{L^2_x} \gec -C\,m\|u_r(t_1)\|_{L^2_x} + \int_{t_1}^{t_2}\|u_r(t)\|_2^2 dt,}
which leads to  blow-up of $\|u_r(t)\|_{2}$ in finite time from the blowup exhibited by~$f'(t)\gec f^2(t)$. 
This concludes the proof of Theorem~\ref{thm:onepass} in the case $\sg=-1$. 

\subsection{Virial estimate in the scattering case $\sg=+1$} \label{subsec:K2pos}
In this case, the sign-definiteness in the variational region becomes more delicate. For simplicity, we make a specific choice\footnote{The important property of $\phi$ in the $\sg=+1$ case is that the convergence as $r\to\I$ is slow.} for $\phi$: 
\EQ{
 \phi(r)=\frac{r}{1+r},}
and rewrite \eqref{eq:OgTsu} in a different way: 
\EQ{ \label{sat-virial sg+}
 \p_t\LR{\phi_m u|iu_r} = 2 K(\chi_m u) + \int\big[\frac{|u|^2}{m^2}f_{3,m} + |u|^4f_{4,m}\big] \,  dx, }
where $\chi_m(r)=\chi(r/m)$ and $f_{j,m}(r)=f_j(r/m)$ are smooth functions defined by 
\EQ{
 \pt \chi(r) = \sqrt{\phi_r}(r) = \frac{1}{1+r}, 
 \pq f_3(r):=-2|\chi_r(r)|^2 + \frac{\phi_{rrr}(r)}{2} = \frac{1}{(1+r)^4}, 
 \pr -f_4(r):= -\big[ \frac{3(\phi_r)^2}{2}-\frac{\phi_r}{2}-\frac{\phi}{r} \big](r) = \frac{r(2r^2+7r+8)}{2(1+r)^4} \simeq \frac{\phi(r)^2}{r}}
In order to absorb the $|u|^4$ term, we use another radial Sobolev inequality \eqref{rad Sob sg+}
\EQ{
  \int_m^\I |u|^4rdr \lec \int_m^\I|u|^2r^2dr \int_m^\I|u_r|^2dr.}
The same argument as in \eqref{weight Sobolev} transforms it into 
\EQ{
 \int_0^\I |u|^4 \phi_m^2 r dr \lec \|u\|_{L^2}^2 \int_0^\I |u_r|^2 \phi_m^2 \, dr.}
Since $(\chi_m')^2\simeq f_{3,m}/m^2$, the right-hand side is estimated by 
\EQ{
 \int_0^\I |u_r|^2 \phi_m^2 dr \lec \int_0^\I (|(\chi_m u)_r|^2 + m^{-2}|u|^2f_{3,m})r^2 \, dr.}
Thus we obtain 
\EQ{ \label{u^4 error bd} 
 \int|u|^4f_{4,m}\, dx \lec m^{-1}\left[ \|\na(\chi_m u)\|_{L^2}^2 + \int \frac{|u|^2}{m^2}f_{3,m} \, dx\right].} 
In particular it is $O(m^{-1})$ since $u$ is bounded in $H^1$ by Lemma \ref{lem:sign}. 

In the hyperbolic region, the cut-off in $K$ has little impact, since by the same expansion as in \eqref{exp K2}, we have 
\EQ{
 \pn K(\chi_mu) \pt= K(Q+\chi_mw+(\chi_m-1)Q)
 \pr = -\la_1\mu\LR{Q|\psi}-\LR{2Q+Q^3|\chi_m\ga+(\chi_m-1)(Q+2\la_1\fy)}
 \prQQ +O(\|w\|_{H^1}^2+\|(\chi_m-1)Q\|_{H^1}^2)
 \pr = -\la_1\mu\LR{Q|\psi}-\LR{2Q+Q^3|\chi_m\ga}+O(\|w\|_{H^1}^2 + e^{-m}),}
thanks to the exponential decay of $Q$. 

In the variational region $d_{Q}(u)>\delta_{*}$, if $\|\na u\|_2\le\mu$ for some small $\mu>0$, then by Gagliardo-Nirenberg \eqref{GN}, we have \EQ{
 K(\chi_mu)\simeq \|\na\chi_mu\|_2^2.} 
Otherwise $\|\na u\|_2>\mu$ and so we have from \eqref{vari I'},
\EQ{ \label{Va-region}
 K(u(t)) \ge \ka_3(\de_*):=\min(\ka_1(\de_*),\mu) \pq (t\in I').}
Hence, if we choose $\e_*>0$ so that 
\EQ{ \label{cond eps}
  \e_*^2<\ka_3(\de_*)/6,}
then $d_Q(u)>\de_*$ and $J(u)<J(Q)+\e^2$ imply 
\EQ{ \label{cond I2}
 I(u) < J(Q) - \ka_3(\de_*)/3.}
Since $I(\chi_m u)\le I(u)$, Lemma \ref{lem:vari} (II) yields
\EQ{ \label{lowbd K2chi}
 K(\chi_m u) \ge \min(\ka_4(\de_*),\ka_0\|\na\chi_mu\|_2^2),}
where $\ka_4(\de_*):=\ka_2(\ka_3(\de_*)/3)$. This bound is valid including the case $\|\na u\|_2\le\mu$ (making $\ka_0$ smaller if necessary). 

Now choose $m>1$ so large that we have 
\EQ{
 m^{-1} \ll \min(\ka_4(\de_*),\ka_0)=:\ka_5(\de_*).}
Then \eqref{u^4 error bd} and \eqref{lowbd K2chi} imply
\EQ{
 2K(\chi_m u)+\int\left[\frac{|u|^2}{m^2}f_{3,m}+|u|^4f_{4,m}\right] \, dx \ge 0.}
Hence we obtain the monotonicity 
\EQ{
 \p_t\LR{\phi_m u|iu_r} \ge 0,}
in the variational region $d_Q(u)\ge \de_*$, i.e., for $t\in I'$, cf.~\eqref{vari I'}. 
By the same argument as in the case of $\sg=-1$, we now arrive at a contradiction by choosing $m\ge 1/R$ and $R_*,\de_*,\e_*>0$ so that \eqref{cond0 Rede*}, \eqref{cond eps} and 
\EQ{ \label{R* cond scat}
 R_*^2 \ll \ka_5(\de_*)}
are satisfied. This concludes the proof of Theorem~\ref{thm:onepass}.

\section{Scattering for $K>0$ solutions}\label{sec:scat}

In this section we establish the following scattering result, following the proof scheme of~\cite{KM1}. Let $R_{*}$ be a fixed choice of $R$ as in
Theorem~\ref{thm:onepass}. 

\begin{prop}\label{prop:scatKpos}
Let $\e_*,R_*>0$ be as in Theorem \ref{thm:onepass}. There exists $\cN<\I$ such that if a solution $u$ of the NLS equation~\eqref{eq:Schr3} on $[0,\I)$ satisfies $M(u)=M(Q)$ and $E(u)\le E(Q)+\e_*^2$, as well as $d_Q(u(t))\ge R_*$ and $\Sg(u(t))=+1$ for all $t\ge 0$, then $u$ scatters to $0$ as $t\to\I$ and $\|u\|_{L^4_t((0,\I);L^4_x)}\le \cN$. 
\end{prop}
\begin{proof}
Let $\UU$ be the collection of all solutions $u$ of \eqref{eq:Schr3} on $[0,\I)$ satisfying 
\EQ{ 
 \pt M(u)=M(Q),\pq
  E(u) \le E(Q)+\e_*^2, \pq d_Q(u[0,\I))\subset [R_*,\I), \pq 
 \pr \Sg(u(t)) =+1 \pq(t\ge 0).}
Indeed, Lemma \ref{lem:sign} implies that $\Sg(u(t))=+1$ is preserved as long as $d_Q(u(t))\ge R_*>2\e$, and $u(t)$ is a forward global solution uniformly bounded in $H^1$. Moreover, the lower bound on $d_Q$ implies, by Lemma \ref{lem:dist L2} that 
\EQ{ \label{L2 dist bd}
 \inf_{t\ge 0} \dist_{L^2}(u(t),\cS_1) \gec \min(R_*,\de_E^2) = R_*.}
For each $E>0$, let $\cN(E)$ be defined as 
\EQ{ \label{def ME}
 \cN(E):=\sup\{\|u\|_{L^4_t((0,\I);L^4_x)} \mid u\in\UU,\ E(u)\le E\}}
See Section~\ref{sec:AppendixA}  for the relevance of $L^{4}_{tx}$. 
We know by \cite{HolR} that $\cN(E)<\I$ for $E<E(Q)$. 
In order to extend this property to $E(Q)+\e_*^2$, put 
\EQ{ \label{def Estar}
 E^\star=\sup\{E>0 \mid \cN(E)<\I\}}
Then $E^{\star}\ge E(Q)$ and assume towards a contradiction that 
\EQ{\label{eq:contraE}
 E^\star=E(Q)+\e^{2}<E(Q)+\e_*^2,} 
where $0\le \e <\e_*$.  
We consider a nonlinear profile decomposition in the sense of Bahouri-G\'erard~\cite{BaG} for any sequence $u_n\in\UU$ satisfying 
\EQ{ \label{def un} 
 E(u_n)\to E^\star, \pq \|u_n\|_{L^4_t((0,\I);L^4_x)}\to\I.} 
We are going to show that the remainder in the decomposition vanishes in a suitable sense and that there is only one profile, which is a critical element, i.e.,
\EQ{ \label{crit elm}
 u_\star\in\UU,\pq E(u_\star)=E^\star, \pq \|u_\star\|_{L^4_t((0,\I);L^4_x)}=\I.}

Before starting the decomposition for $u_n$, we translate $u_n$ in $t$ to achieve 
\EQ{ \label{eq:Kunmu}
 d_Q(u_n(0))\ge\de_X, \pq K(u_n(0)) > 6\e_*^2.}
Since $d_Q(u_n(t))$ will never come down to $R_*$, the ejection Lemma \ref{lem:eject} gives $0<T_n\lec \log(\de_X/R_*)$ so that $d_Q(u_n(T_n))\ge\de_X$. Since $\Sg=+1$, we have a uniform $H^1$ bound on $u_n$ and so, in view of Lemma \ref{lem:simplescat} together with $\|u_n\|_{L^4_{t,x}(0,\I)}\to\I$, we have $\|\na u_n(T_n)\|_2>\mu$. Hence by Lemma \ref{lem:vari}, 
\EQ{
 K(u_n(T_n))\ge \min(\ka_1(\de_X),\ka_0\mu^2)\ge \ka_3(\de_*)>6\e_*^2,}
where $\ka_3$ is defined in \eqref{Va-region} and we used the condition \eqref{cond eps} on $\e_*$. Translating $u_n:=u_n(t-T_n)$, we may assume \eqref{eq:Kunmu}.

Now apply Proposition~\ref{prop:BG} to the  sequence $\{u_n(0)\}$. This yields, cf.~\eqref{eq:linprofile}
\EQ{\label{eq:profdec}
e^{-it\Delta} u_n(0) = \sum_{0\le j<k} e^{-i(t+t_{n}^{j})\Delta}  v^j + \ga_n^k(t)
}
We take here $k$ sufficiently large so that $\ga_{n}^{k}$ is  small in the sense of~\eqref{orth}; it will always be assumed that $n$ is large. 
The first step consists in showing that due to~\eqref{eq:contraE} one has $v^{j}=0$ for all but one $j$ as well as
\EQ{
\lim_{k\to\I} \limsup_{n\to\I} \|\ga_{n}^{k}\|_{L^{\I}_{t}H^{1}} =0
}
By the partition property \eqref{H1 orth} one has 
\EQ{\label{eq:mass decomp}
M(u_{n}) &= \sum_{0\le j<k} M(v^{j}) + M(\ga_{n}^{k}) + o(1) 
}
as $n\to\I$.   Assume that $M(v^{0})>0$ and $M(v^{1})>0$. Then $M(v^{j})<M(Q)$ for each $j$. 
From \eqref{eq:Kunmu}, $J(u_n)<J(Q)+\e_*^2$, and \eqref{H1 orth}, we infer that 
\EQ{\nn 
 \pt J(Q)-\e_*^2>  \limsup_{n\to\I} [J(u_n)-K(u_n(0))/3] 
 \pr\ge \limsup_{n\to\I} G(u_n(0)) =  \sum_{j<k}G(v^j)+\limsup_{n\to\I} G(\ga_n^j),}
where we used that $G(\fy)\simeq \|\fy\|_{H^1}^2$ is conserved by the linear flow. 
Since $G$ is positive definite, we obtain $G(v^j)<J(Q)-\e_*^2$ for all $j$, 
 which implies, via the minimizing property \eqref{cnstr min} and the invariance of~$G$ under the linear flow, 
 that $K(e^{-it\Delta}v^j)\ge 0$ for all $t\in\R$. 
 Let $t_n^j\to t_\I^j\in[-\I,\I]$ and let $u^j$ be the nonlinear profile associated with $v^{j}$, i.e., that solution of~\eqref{eq:Schr3}  satisfying 
\EQ{ \label{v-u at 0}
 \lim_{t\to t_\I^j}\| u^j(t)-  e^{-it\Delta}v^j\|_{H^{1}}=0,}
which exists at least locally around $t=t_\I^j$, either by solving the Cauchy problem at $t_\I^j\in\R$ or by applying the wave 
operator at $t_\I^j\in\{\pm\I\}$.  
By the preceding, 
\EQ{\label{eq:Muj}
M(u^{j})=M(v^{j})<M(Q),\quad K(u^{j}(t))\ge 0 \quad\forall\; t\text{\ near\ }t^{j}_{\I}
}
In particular,  $E(u^{j})\ge0$. 
We have the following partition of the nonlinear energy
\EQ{ \label{eq:ene decomp}
 E(u_{n}) &= \sum_{0\le j<k} E(u^{j}) + E(\ga_{n}^{k}(0)) + o(1)
}
as $n\to\I$. Since 
\[
\| u\|_{4}^{4} \lec \|\nabla u\|_{2}^{2} \| u\|_{3}^{2}
\]
and \eqref{orth} 
imply that $E(\ga_{n}^{k})\simeq \|\ga_{n}^{k} \|_{H^{1}}^{2}$ provided $k$ is large, we conclude that
\EQ{\label{eq:Euj}
0\le E(u^{j}) \le \lim_{n\to\I}E(u_{n}) = E(Q)+\e^{2} \quad \forall\; j.
}
If $E(u^{j})<E(Q)$ for all $j$, then we conclude by the preceding and~\cite{HolR}
that each $u^{j}$ exists globally and scatters with $\| u^{j}\|_{S(\R)}<\infty$,  where 
\EQ{
 S(I) := L^\I_tH^1_x(I) \cap L^2_t W^{1,6}_x(I)} 
is the norm of the Strichartz spaces.
 But then one has the following {\it nonlinear profile decomposition}, 
 \EQ{\label{eq:nonlinprofile}
 u_n = \sum_{j<k} u^j_n+\ga_n^k+\err_{n}^{k},\pq u^j_n:=u^j(t+t_n^j)}
 where $\ga_{n}^{k}$ are as in~\eqref{eq:profdec}, and the errors $\err_{n}^{k}$ satisfy
\EQ{\label{eq:errnk}
\lim_{k\to\infty} \limsup_{n\to\I} \| \err_{n}^{k} \|_{S_0(\R)} =0
} 
where 
\EQ{
 S_0(I) = L^\I_tL^2_x(I) \cap L^4_{t,x}(I)}
is a subcritical Strichartz norm. To see this, first let 
$w_n^j$ be the nonlinear solution with the initial condition $w_n^j(0)=e^{-it_n^j\De}v^j$. By the local theory
\EQ{\nn 
 w_n^j(t)-u^j(t+t_n^j)\to 0 \pq(n\to\I)}
in $S(I)$ for some interval $I\ni 0$. 
Now apply Lemma~\ref{PerturLem} with $$t_{0}=0, \quad u=u_{n}, \quad w=\sum_{j<k} u^j_n+\ga_n^k$$ to conclude~\eqref{eq:nonlinprofile}. 
In order to apply this lemma, one needs to verify that  the nonlinear interactions between  all $\{u^{0}, u^{1}, \ldots, \ga_{n}^{k}\}$, as well as $|\ga_{n}^{k}|^{2} \ga_{n}^{k}$, 
vanish in the $L^{\frac85}_{t}L^{\frac43}_{x}$-norm in the limits as~$n\to\I$ and then $k\to\I$. 
However, as usual, this follows by expanding the cubic nonlinearity and using the Strichartz control available for each of these functions.  
Since~\eqref{eq:nonlinprofile} contradicts~\eqref{def un}, we must have that at least one $u^{j}$, and therefore exactly one, say $u^{0}$, satisfies $E(u^{0})\ge E(Q)$.
But then $E(u^{j})\lec \eps^{2}\ll E(Q)$ for all $j\ge1$ and  these $u^{j}$ are globally defined and scatter. If $u^0$ also scatters, then by the same reasoning we get a contradiction. Hence $u^0$ does not scatter either for $t\to\I$ or for $t\to-\I$. 

Let $\tilde u^{0}(t,x)=\alpha u^{0}(\al^2 t,\alpha x)$ be the rescaling of $u^{0}$ such that $M(\tilde u^{0})=M(Q)$. Then
\EQ{\label{eq:Etildeuj}
 E(\tilde u^{0})\le \frac{M(u^{0})}{M(Q)}(E(Q)+\e^{2}),} 
so we have $M(u^{0})>M(Q)-O(\e^{2})$ for otherwise one has global scattering for $u^0$ from~\cite{HolR} and therefore a contradiction via~\eqref{eq:nonlinprofile}. Hence $\al=1+O(\e^2)$. 
In view of~\eqref{eq:mass decomp} and \eqref{eq:ene decomp}, $J(u^{j})\lec \e^{2}$ for all $j\ge1$. Since
\[
 J(u^{j})\ge G(u^{j}(t)) \ge \|u^{j}(t)\|_{H^1}^{2}/6 
\]
for all $t$ and any $j\ge1$, we conclude that $\| u^{j}\|_{L^{\I}H^{1}}\lec \e$ for each $j\ge1$. In fact, the same logic with the asymptotic orthogonality yields the following stronger bound
\EQ{\label{eq:sumujsmall}
  \sum_{j\ge1} \|u^{j}\|_{L^{\I}_tH^{1}_x}^{2}+\sup_{k\ge 1}\limsup_{n\to\I}\|\ga_n^k\|_{L^\I_tH^1_x}^2\lec \e^{2},
}
which will be crucial later.

If $t_{\I}^{0}=\I$, then by the definition $u^0$ scatters as $t\to\I$, and so one obtains a contradiction as before via \eqref{eq:nonlinprofile}. 

If $t_{\I}^{0}=-\I$, then $\ti u^{0}$ scatters as $t\to-\I$, 
and we face the following dichotomy: 
either $\ti u^{0}$ satisfies $d_{Q}(\ti u^{0}(t))> 2\e_*$ on its entire interval of existence, or there exists some time $t_{*}$ in the interval of existence of $u^0$ at which 
\EQ{ \label{u^0 comes to Q}
  d_{Q}(\ti u^{0}(\al^{-2}t_{*}))\le 2\e_*.} 
In the former case, we infer that $\Sg(\ti u^{0}(t))=+1$ is preserved from $t=-\I$, therefore $\ti u^{0}$ exists globally and satisfies $d_{Q}(\ti u^{0}(t))>R_{*}$ for large $t>T'$; 
for this last statement one invokes Lemma~\ref{lem:eject} to obtain ejection, if needed. 
Then $\ti u^0(t-T')$ is a critical element, since $\ti u^0$ does not scatter for $t\to\I$. In the latter case \eqref{u^0 comes to Q}, we have 
\[ 
 \|u^{0}\|_{L^{4}((-\I,t_{*}+C];L^{4}_{x})}<\I \text{\ with\ $C\gec 1$,}\] 
and thus also~\eqref{eq:nonlinprofile} on the time interval $(-\I,t_{*}-t^{0}_{n})$. But then, by~\eqref{eq:sumujsmall},  
\[
 \dist_{L^2}(u_{n}(t_{*}-t^{0}_{n}),\cS_1) 
 \le d_Q(\ti u^{0}(\al^{-2}t_{*}))+ O(\e) \lec \e_* \ll R_*,
\]
and $t_*-t_n^0\to\I$, contrary to \eqref{L2 dist bd} for $u_n$. 

The only remaining case is $t^{0}_{\I}\in \R$. In that case we use the general fact that the nonlinear
profile decomposition~\eqref{eq:nonlinprofile} holds locally around~$t=0$. In addition to that, we have by \eqref{eq:Kunmu} and \eqref{eq:sumujsmall}, 
\EQ{
 \pt d_Q(\ti u^0(\al^{-2}t_\I^0))\ge \limsup_{n\to\I} d_Q(u_n(0)) - O(\e) > \de_S,}
and so $\Sg(\ti u^0(\al^{-2}t_\I^0))=+1$ since $K(\ti u^0(\al^{-2}t_\I^0))=\al K(u^0(t_\I^0))\ge 0$. If $u^0$ scatters as $t\to\I$, then we obtain a contradiction as before via \eqref{eq:nonlinprofile}. Hence $\ti u^0$ does not scatter for $t\to\I$.  
If $d_{Q}(\al^{-2}\ti u^{0}(t))>2\e_*$ for all $t\ge t_\I^0$ on the maximal
interval of existence of~$\ti u^{0}$, then $\Sg(\ti u^{0})=+1$ for these $t$, which implies that $\ti u^0$ is forward global and a critical element after time translation (use the ejection Lemma \ref{lem:eject} if needed to conclude that $d_{Q}(\ti u^{0}(t))\ge R_{*}$ for large $t$). 
Otherwise, there exists $t_{*}>t_\I^0$ minimal so that $d_{Q}(\ti u^{0}(\al^{-2}t_{*})<2\e_*$. But then~\eqref{eq:nonlinprofile} remains valid near~$t_{*}$ and one again obtains a contradiction to
 $$\dist_{L^2}(u_{n}(t_*-t_\I^0),\cS_1)\gec R_{*}\gg\e_*.$$

Therefore, the conclusion is that $u^0$ is a critical element after time translation, $E(u^{0})=E^{\star}$, $M(u^{0})=M(Q)$, and so $v^{j}=0$ for all $j\ge1$ as well as $\ga_{n}^{k}\to0$ in the energy sense as $n\to\I$. Hence (after extracting a subsequence)
\EQ{
 \lim_{n\to\I} \|u_n(T_n)-u^0(t_n^0)\|_{H^1} = 0,}
where $T_n$ is the time shift for \eqref{eq:Kunmu}. Both $T_n$ and $t_n^0$ are bounded above for $n\to\I$. If $t_n^0\to-\I$ then $u^0$ scatters as $t\to-\I$, and the local theory of the wave operator implies that $\|u_n\|_{L^4_{t,x}(-\I,T_n)}\to 0$. 

Applying the above result to the sequence $u_n:=u^{0}(t+\tau_{n})$ for arbitrary $\tau_{n}\to\I$, one now concludes that the forward trajectory $\{u^{0}(t)\}_{t\ge0}$ is precompact in~$H^{1}$, since $t_n^0\to-\I$ would imply that $\|u^0\|_{L^4_{t,x}(-\I,\tau_n)}\to 0$ which is a contradiction. 

Finally,  integrating  the saturated virial identity~\eqref{sat-virial sg+} from Section~\ref{subsec:K2pos} between some  positive time and $t=\I$ now proves that such a critical element cannot exist. Note that $K(\chi_mu)$ has a positive lower bound in the variational region (for large $m$) by the precompactness of the forward trajectory of the critical element. 
This shows that $E^\star<E(Q)+\e_*^{2}$ is impossible, concluding the proof. 
\end{proof}

\section{The proof of Theorem \ref{thm:main}}\label{sec:proofmain}

Let $\HH^{\e}$ and $\HH^{\e}_{\alpha}$ be as defined in~\eqref{eq:HHe}, \eqref{eq:HHea}, respectively. 
We  introduce the following subsets according to the global behavior of the solution $u(t)$ of~\eqref{eq:Schr3}: for $\si=\pm$ respectively, 
\EQ{\label{eq:STB def}
 \pt \cS_\si^{\e}=\{u(0)\in\HH^\e_{1} \mid \text{$u(t)$ scatters as $\si t\to\I$}\},
 \pr \T_\si^{\e}=\{ u(0)\in\HH^\e_{1} \mid   \text{$u(t)$ trapped by $\cS_{1}$ for $\si t\to\I$}\}, 
 \pr \B_\si^\e=\{ u(0)\in\HH^\e_{1} \mid   \text{$u(t)$ blows up in $\si t>0$}\}.}
The trapping for $\T_+^\e$ can be characterized as follows, with any $R\in(2\e,R_*)$: 
\EQ{\nn 
 \exists T>0, \pq \forall t>T,\pq d_Q(u(t))<R.}
Obviously those sets are increasing in $\e$, and have the conjugation property
\EQ{\nn 
 X_\mp^\e = \{\bar{u}(0)\in\HH^\e_{1} \mid  u(0)\in X_\pm^\e\},}
for $X=\cS,\T,\B$. Moreover, $\cS_+$ and $\T_+$ are forward invariant by the flow of NLS, while $\cS_-$ and $\T_-$ are backward invariant. 
By what we have done so far 
\EQ{\nn 
 \HH^\e_{1} = \cS_+^\e \cup \T_+^\e \cup \B_+^\e = \cS_-^\e \cup \T_-^\e \cup \B_-^\e,}
with the union being disjoint for each sign. 
It follows from the scattering theory that $\cS_\pm^\e$ are open (relatively, in~$\HH^{\e}_{1}$). 
We claim the same for $\B_\pm^\e$, which is not a general fact. 
Thus, suppose that $u(t)$ blows up at $0<T<\I$. This is equivalent to $\|\nabla u(t)\|_{2}\to\I$ as $t\to T-$. Since
$$
K(u(t))=3E(u) - \frac12 \|\nabla u(t)\|_{2}^{2}
$$
it follows that $K(u(t))\to-\infty$ as $t\to T-$. We now claim that if $v(0)\in \HH^{\e}_{1}$ with $\|v(0)-u(T')\|_{H^{1}}<1$ where $T'<T$ is fixed and very close to~$T$,
then $v(0)$ leads to a solution $v(t)$ which blows up in finite time. Suppose not. Then from the energy constraint $E(v(0))<E(Q)+\e^{2}$ we know that $K(v(t))$ can only
change sign by coming very close to $\cS_{1}$. But since~\eqref{eq:Knegmonotone}, with $m\gg1$ fixed, implies that $\LR{\phi_{m}v|iv_{r}}$ decreases
on any time  interval where $K(v(t))\le -1$, and since  $\LR{\phi_{m}v|iv_{r}}=0$ on $\cS_{1}$, it follows that $K(v(t))\le -1$ for all $t\ge0$. 
But then we obtain a contradiction via~\eqref{eq:Knegmonotone}. Thus $v$ blows up in finite time as claimed. 
Therefore,  $\B_\pm^\e$ are also open, so $\T_\pm^\e$ are relatively closed in $\HH^\e_{1}$.  

Since $\B^\e_+$ and $\cS_+^\e$ are disjoint open, they are separated by $\T_+^\e$, that is, any two points from $\cS_+^\e$ and $\B_+^\e$ cannot be joined by a curve in $\HH^{\e}_{1}$ without passing through~$\T_+^\e$. 

In a small ball around $\cS_{1}$, it is easy to see by means of the linearized flow that the open intersections $\B^\e_\pm\cap \cS_\mp^\e$, $\B^\e_+\cap \B^\e_-$ and $\cS_+^\e\cap \cS_-^\e$ are all non-empty for any $\e>0$. Examples of  solutions belonging to the first set are given by data in~\eqref{eq:lam+-}
\EQ{\label{eq:BSsol}
 \la_+(0)=-\la_-(0)=\e\rho, \pq \ga(0)=\al Q',}
 with real parameters, whence 
 \EQ{
 u(0) &= Q + \e \rho \GZ_{+} - \e \rho \GZ_{-}  + \al Q'\\
 & = Q + \al Q' -2i  \e \rho  \psi}
and thus
\EQ{
M(u) &= M(Q + \al Q') + 4\e^{2} \rho^{2} M(\psi) \\
 & = M(Q) + \al \LR{Q|Q'} + \al^{2} M(Q') + 4\e^{2}\rho^{2} M(\psi)
 }
where $|\rho|\ll 1$, and $\al>0$ should be chosen such that $M(u)=M(Q)$. 
Then the linearized flow of $\lambda_{1}(t)$ is $\lambda_{1}^{(0)}(t)=\e\rho \sinh(\mu t)$ and $\|\ga^{(0)}\|_2=O(\e^2\rho^2)$. In fact, the estimates \eqref{eq:eject_est}
show that the true $\lambda_{1}(t),\ga(t)$ deviate from these only by quadratic corrections $O(\e^2\rho^2)$. 
Therefore, at exit time from the $\delta_{X}$-ball one has $t \Sg(u(t))$ of a fixed sign. 
Hence, choosing the sign of $\rho$ correctly leads to solutions  $u\in \B^\e_-\cap \cS_+^\e$, or $\tilde u\in \B^\e_+\cap \cS_-^\e$, respectively. 

The analogous construction with $\pm \cosh(\mu t)$ instead of $\sinh(\mu t)$ 
furnishes examples of solutions $u_{\pm}$ belonging to 
 $\B_+^\e\cap \B_-^\e$ and $\cS_+^\e\cap \cS_-^\e$, respectively. It is clear that these constructions actually give open nonempty 
 sets of solutions relative to~$\HH^{\e}_{1}$ (indeed, we can perturb $\ga(0)$ within $O(\e^2\rho^2)$). 
 
 Next, note that by construction $\|u(t)-u_{+}(t)\|_{H^1}\ll \e$ while $d_{Q}(u(t))\gg \e$ for some large times (but before exit from the $\delta_{X}$-ball). 
It follows that we may connect $u(t)$ with $u_{+}(t)$ by a curve segment $\Gamma$ within $\HH^{\e}_{1}$ and within the set $\Sg<0$ and $d_Q(u)\gg\e$. 
Since $u(t)\in\cS_{-}^{\e}$ and $u_{+}(t)\in \B_{-}^{\e}$, there exists $p_{0}\in \T_{-}^{\e}\cap\Ga$. Since the solution starting from $p_{0}$ enters the $3\e$-ball around $\cS_{1}$ as $t\to-\I$, and initially $p_{0}$ is much further away and also $\Sg(p_{0})=-1$, we conclude by the one-pass theorem that $p\in \B^\e_+$. Hence $\T^\e_-\cap \B^\e_+$ is non-empty as well. 
 In the same way, we can find a point on the curve connecting $u(t)$ and $u_{-}(t)$ for some $t<0$, which is in $\T^\e_+\cap \cS_-^\e$. 
 Therefore, $\T^\e_\pm\cap \B^\e_\mp$ and $\T^\e_\pm\cap \cS_\mp^\e$ are both not empty. Taking the limit $\e\to+0$, it is easy to observe 
 that they contain infinitely many points on different energy levels. 

The sets $\T^{\e}_{+}\cap \T^{\e}_{-}$ contain all of $\cS_{1}$, and are therefore not empty. By considering curves on the hyperplane $\{\Im u=0\}$ connecting $u_{+}(0)$ with $u_{-}(0)$ (the solutions from before) and which are disjoint from~$\cS_{1}$, we obtain infinitely many points in~$\T^{\e}_{+}\cap \T^{\e}_{-}\setminus\cS_{1}$. 

By a simple scattering theory argument one can check that $\cS_{\pm}^{\e}$ is path-wise connected, as is every slice of fixed energy of this set 
(all relative to~$\HH^{\e}_{1}$). See~\cite{NakS} (7.15)--(7.19) for details. Therefore, $\cS_{+}^{\e}$ is its own connected component. 
To find a curve connecting $0$ to $\I$ in $H^{1}$ within
that set, follow  a solution in $\B_{-}^{\e}\cap \cS_{+}^{\e}$ to blow-up time. This concludes the proof of Theorem~\ref{thm:main}.

\section{Construction of the center-stable manifold in the energy topology}
\label{sec:mf}

We construct a center-stable manifold containing the ground state $Q$.  
All function spaces will be radial. Moreover, $B_\delta(Q)$ denotes a $\delta$-ball in the energy space centered at~$Q$.  In contrast to Section~\ref{sec:ground}, 
we work with the
matrix formalism usually employed in asymptotic stability theory; the latter is of course closely related
to the formalism of Section~\ref{sec:ground} but since we build upon~\cite{S}, \cite{ES}, \cite{Bec1}, \cite{Bec2} it is convenient to
adopt the complex-linear point of view. In what follows, 
\begin{equation}\label{eq:H0}
 \calH(\alpha,\ga) = \left[ \begin{matrix} 
 -\Delta + \alpha^2 - 2Q^2(\cdot,\alpha) & -  e^{2i\gamma}  Q^2(\cdot,\alpha) \\
   e^{-2i\gamma}  Q^2(\cdot,\alpha) &  \Delta - \alpha^2 + 2Q^2(\cdot,\alpha)
 \end{matrix} \right]
\end{equation}
The following proposition constructs the center-stable manifold in a small neighborhood of~$Q$. It should be
compared to Definition~\ref{def:scattoQ}; in fact, it provides much more detailed information than
what is required by that definition. In Remark~\ref{rem:extendM} we extend  $\M$ so as to cover all of~$\cS$, and 
Corollary~\ref{cor:paraM} characterizes the stable manifold, which lies in~$\M$. A word about notation: henceforth, $\ga$ plays
the role of a phase which has nothing to do with its previous usage, cf.~\eqref{eq:BSsol}. 

For results on asymptotic stability analysis in the subcritical, and thus orbitally stable case, see Buslaev, Perelman~\cite{BP1}, \cite{BP2}, Cuccagna~\cite{Cuc}, 
and Soffer, Weinstein~\cite{SofWei1}, \cite{SofWei2}. See also Pillet, Wayne~\cite{PW}. 

\begin{prop}\label{prop:mf}
There exists $\delta>0$ small and a smooth manifold $\M\subset H^{1}_{\mathrm{rad}}$ with the following properties: $\cS\cap B_\delta(Q)\subset \M\subset B_{\delta}(Q)$, 
$\M$ divides $B_{\delta}(Q)$ into two connected components, and 
any initial data $u_0\in\M$ generates a solution of~\eqref{eq:Schr3} for all $t\ge0$ of the form
\begin{equation}
\label{eq:psiv}
u(x,t) = e^{i\theta(t)} Q(x,\alpha(t)) + v(x,t)\quad \forall \; t\ge0
\end{equation}
where $\theta(t)=-\int_0^t \alpha^2(s)\, ds + \gamma(t)$, 
\EQ{\label{eq:dotpath}
\| \dot\gamma\|_{L^1\cap L^\infty(0,\I)} + \| \dot \alpha\|_{L^1\cap L^\infty(0,\I)} &\lec\delta^2,\\
\sup_{t\ge0} [|\alpha(t)-1|+|\gamma(t)|]\lec \delta 
}
The function $v$ is small in the sense
\begin{equation}
\| v\|_{L^\infty_t ((0,\I);H^1(\R^3))} + \|v\|_{L^2_t((0,\I); W^{1,6}(\R^3))} \lec \delta
\label{eq:v_disp}
\end{equation}
and it scatters:  $v(t)=e^{-it\Delta} v_{\I}+o_{H^{1}}(1)$ as $t\to\I$ for a unique $v_{\I}\in H^{1}$. 

$\M$ is unique in the following sense: there exists a constant $C$ so that any $u_0\in B_\delta(Q)$ satisfies $u_{0}\in\M$ if and
only if the solution $u(t)$ of~\eqref{eq:Schr3} with data $u_{0}$ has the property that 
$\dist(u(t),\cS_{1})\le C\delta$
 for all $t\ge0$. 
\end{prop}
\begin{proof}
Inserting \eqref{eq:psiv} into~\eqref{eq:Schr3} yields
\begin{equation}
\label{eq:Rvsystem}
i\partial_t \Rv + \wt\calH(t) \Rv = \dot{\gamma}(t) \wt\xi(t) -i\dot{\alpha}(t) \wt \eta(t) + \wt N(t,v,\bar v)
\end{equation}
where 
\begin{equation}\label{eq:calH}
\wt \calH(t) = \left[ \begin{matrix} 
-\Delta - 2Q^2(\cdot,\alpha(t)) & -e^{2i\theta(t)} Q^2(\cdot,\alpha(t)) \\
 e^{-2i\theta(t)} Q^2(\cdot,\alpha(t)) &  \Delta + 2Q^2(\cdot,\alpha(t))
\end{matrix} \right]
\end{equation}
as well as 
\begin{equation}\label{eq:wtxiwteta}
\wt \xi(t) = \binom{e^{i\theta(t)}Q(\cdot,\alpha(t))}{-e^{-i\theta(t)}Q(\cdot,\alpha(t))},\qquad \wt \eta(t) = 
 \binom{e^{i\theta(t)}\partial_\alpha Q(\cdot,\alpha(t))}{e^{-i\theta(t)}\partial_\alpha Q(\cdot,\alpha(t))}
\end{equation}
and 
\begin{equation}\label{eq:nonlinearity}
\wt N(t,v,\bar v)= \binom{2e^{i\theta(t)}Q(\cdot,\alpha(t))|v|^2 + e^{-i\theta(t)}Q(\cdot,\alpha(t))v^2 + |v|^2 v}{-2e^{-i\theta(t)}Q(\cdot,\alpha(t))|v|^2 - e^{-i\theta(t)}Q(\cdot,\alpha(t))\bar{v}^2 - |v|^2 \bar{v}}
\end{equation}
Next, set  
\EQ{\label{eq:vwswitch}
v(t)=e^{i\theta_0(t)} w(t),\qquad \theta_0(t)=-\int_0^t \alpha^2(s)\, ds.
}
  Then \eqref{eq:Rvsystem} turns into
\begin{equation}
\label{eq:Rvsystem2}
i\partial_t W + \calH_{\pi}(t) W = \dot{\gamma}(t)\xi(t) -i\dot{\alpha}(t)\eta(t) + N_{\pi}(t,W)
\end{equation}
where $W=\Rw$, and with $\pi=(\al,\ga)$, 
\begin{equation}\label{eq:Ht}
 \calH_{\pi}(t) = \left[ \begin{matrix} 
-\Delta +\alpha^2(t)  - 2Q^2(\cdot,\alpha(t)) & -  e^{2i\gamma(t)} Q^2(\cdot,\alpha(t)) \\
   e^{-2i\gamma(t)} Q^2(\cdot,\alpha(t)) &  \Delta - \alpha^2(t)  + 2Q^2(\cdot,\alpha(t))
\end{matrix} \right]
\end{equation}
as well as 
\begin{equation}\label{eq:xieta}
\xi(t) = \binom{ e^{i\gamma(t)}Q(\cdot,\alpha(t))}{- e^{-i\gamma(t)}Q(\cdot,\alpha(t))},\qquad \eta(t) = 
 \binom{ e^{i\gamma(t)}\partial_\alpha Q(\cdot,\alpha(t))}{ e^{-i\gamma(t)}\partial_\alpha Q(\cdot,\alpha(t))}
\end{equation}
and 
\begin{equation}\label{eq:NtW}
 N_{\pi}(t,W)= \binom{2 e^{i\gamma(t)} Q(\cdot,\alpha(t))|w|^2 +  e^{-i\gamma(t)} Q(\cdot,\alpha(t))w^2 + 
|w|^2 w}{-2  e^{-i\gamma(t)} Q(\cdot,\alpha(t))|w|^2 -  e^{i\gamma(t)} Q(\cdot,\alpha(t))\bar{w}^2 - |w|^2 \bar{w}}
\end{equation}
At this point we remark that all manipulations which we perform in this proof on \eqref{eq:Rvsystem} and~\eqref{eq:Rvsystem2}
preserve the ``admissible'' subspace $\{\binom{f}{\bar f}\mid f:\R^{3}\to \C\}$. This is necessary in order to return to the scalar formulation~\eqref{eq:psiv}.
In other words, the second row of these systems can be viewed as redundant, as it is always the complex conjugate of the first. 

\noindent Let $\sigma_3=\left[ \begin{matrix}
1&0\\0&-1
\end{matrix}\right]$ be the third Pauli matrix, and set $\xi^*(t)=\sigma_3\xi(t)$, $\eta^*(t)=\sigma_3 \eta(t)$. 
Impose the orthogonality conditions\footnote{Henceforth, we write $\langle\cdot,\cdot\rangle$ for the standard inner product in $L^{2}(\R^{3};\C^{2})$, whereas
$\LR{\cdot|\cdot}$ is the inner product from Section~\ref{sec:ground}.}
\begin{equation}\label{eq:orthcond}
 \lan W(t),\xi^*(t)\ran = 0,\qquad \lan W(t),\eta^*(t) \ran =0\qquad \forall t\ge0
\end{equation}
Note  that this imposes a condition on the data at $t=0$. However, by the inverse function theorem there is a unique choice of 
$\alpha(0)$ and $\gamma(0)$ in a $\delta$-neighborhood of $(1,0)$ so that~\eqref{eq:orthcond} is satisfied; the needed nondegeneracy here is provided 
by  $\langle Q|\partial_\alpha Q\rangle\ne0$. 
Since $\calH_{\pi}(t)^* \xi^*(t)=0$ and $\calH_{\pi}(t)^*\eta^*(t)=-2\al \xi^{*}(t)$,  as well as $\lan \xi(t),\xi^*(t)\ran = \lan \eta(t),\eta^*(t)\ran =0$,
and $\lan \xi(t),\eta^*(t)\ran = \lan \eta(t),\xi^*(t)\ran = 2 \lan Q|\partial_\alpha Q\ran \ne0$, one obtains from~\eqref{eq:Rvsystem2}
that 
\begin{equation}\label{eq:modeq}
\begin{aligned}
\dot{\gamma}(t) \lan \xi(t),\eta^*(t)\ran &= -i \lan W(t), \dot{\eta}^*(t) \ran  -\lan N_{\pi}(t,W),\eta^*(t)\ran   \\
-i\dot{\alpha}(t)  \lan \eta(t),\xi^*(t)\ran &= -i \lan W(t), \dot{\xi}^*(t) \ran  -\lan N_{\pi}(t,W),\xi^*(t)\ran 
\end{aligned}
\end{equation}
The system \eqref{eq:Rvsystem2}, \eqref{eq:orthcond}, \eqref{eq:modeq} determines the evolution of $v(t),\alpha(t),\gamma(t)$ in~\eqref{eq:psiv}. 
In fact, it suffices for \eqref{eq:orthcond}  to hold at one point, say $t=0$ since it then holds for all $t\ge0$. 
More precisely,  one needs to find a fixed point to this system consisting of a path $\pi(t)=(\al(t),\ga(t))$ as well as a function $\Rv$, or equivalently,  $W$ 
satisfying the system as well as the bounds~\eqref{eq:dotpath}, \eqref{eq:v_disp}.   

We begin with the stability part of the underlying contraction argument, i.e., we 
turn \eqref{eq:dotpath} and~\eqref{eq:v_disp} into bootstrap assumptions and then  recover them from this system. 
Thus, suppose $\pi_{0}=(\al_{0},\ga_{0})$ and $W_{0}$ are given so that \eqref{eq:dotpath} and~\eqref{eq:v_disp} hold and consider the following 
system of differential equations:
\EQ{\label{eq:driver}
i\partial_t W + \calH_{\pi_0}(t) W &= \dot{\gamma}(t)\xi_{0}(t) -i\dot{\alpha}(t)\eta_{0}(t) + N_{\pi_0}(t,W_{0}) \\
\dot{\gamma}(t) \lan \xi_{0}(t),\eta_{0}^*(t)\ran &= -i \lan W(t), \dot{\eta}_{0}^*(t) \ran  -\lan N_{\pi_0}(t,W_{0}),\eta_{0}^*(t)\ran   \\
-i\dot{\alpha}(t)  \lan \eta_{0}(t),\xi_{0}^*(t)\ran &= -i \lan W(t), \dot{\xi}_{0}^*(t) \ran  -\lan N_{\pi_0}(t,W_{0}),\xi_{0}^*(t)\ran \\
 \lan W(0),\xi_0^*(0)\ran &= 0,\quad \lan W(0),\eta_0^*(0) \ran =0
}
 where $\calH_{\pi_0}$, $\xi_{0},\eta_{0}$ and $N_{\pi_0}(t,W_{0})$ are defined as above but relative to  the given functions $\pi_{0}, W_{0}$. 
 The initial conditions are $\al(0)=\al_{0}(0)$, $\ga(0)=\ga_{0}(0)$; in addition to the final equation in~\eqref{eq:driver}, $W(0)$ needs to satisfy 
 a further codimension-$1$ condition which will be specified below. 
 
We begin with the $\dot\al, \dot\ga$ part of~\eqref{eq:driver}. The $W$ appearing on the right-hand side will be seen later to satisfy~\eqref{eq:v_disp}; for the moment, we will simply assume
this bound. 
To be more specific, rewrite~\eqref{eq:dotpath}
and~\eqref{eq:v_disp} in the form
\begin{equation}\label{eq:C0C1}\begin{aligned}
\| \dot\gamma\|_{L^1\cap L^\infty} + \| \dot \alpha\|_{L^1\cap L^\infty} &\le C_0\, \delta^2 \\
\| v\|_{L^\infty_t H^1(\R^3)} + \|v\|_{L^2_t W^{1,6}(\R^3)} & \le C_1\,  \delta 
\end{aligned}
\end{equation}
and assume that $C_0\gg C_1^2$. Inserting these bounds in the right-hand side of~\eqref{eq:modeq} yields
\begin{align*}
 \|\dot{\gamma}\|_{L^1\cap L^\infty} + \|\dot{\alpha}\|_{L^1\cap L^\infty} & \lec C_0 C_1\, \delta^3 + C_1^2\, \delta^2  \ll C_0\,\delta^2
 \end{align*}
 provided $\delta$ is small. 
One can thus recover~\eqref{eq:C0C1}. 
The bound on~$v$ (or $W$) is more delicate. 
Since we are in the unstable regime, \eqref{eq:Rvsystem2} is exponentially unstable. More precisely, write
\[
\calH_{\pi_{0}}(t) = \calH_{0} + a_{0}(t) \sigma_3 + D_{0}(t)
\]
with the constant coefficient operator  $\calH_{0}=\calH(\alpha_{0}(0),\ga_{0}(0))$, see~\eqref{eq:H0}, 
and $a_{0}(t)=\alpha_{0}^2(t)-\alpha_{0}^2(0) $, as well as $D_{0}(t)$ equaling 
\[
 \left[ \begin{matrix} 
 - 2(Q^2(\cdot,\alpha_{0}(t)) -Q^2(\cdot,\alpha_{0}(0))) \!\!\!\! &\!\!\!\! -  e^{2i\gamma_{0}(t)}Q^2(\cdot,\alpha_{0}(t)) +  e^{2i\gamma_{0}(0)} Q^2(\cdot,\alpha_{0}(0)) \\
  e^{-2i\gamma_{0}(t)}  Q^2(\cdot,\alpha_{0}(t)) -   e^{-2i\gamma_{0}(0)} Q^2(\cdot,\alpha_{0}(0))\!\!\!\!  & \!\!\!\!   2(Q^2(\cdot,\alpha_{0}(t))-Q^2(\cdot,\alpha_{0}(0)))
 \end{matrix} \right]
\]
Note that $\|a_{0}(\cdot)\|_\infty\lec \delta^2$ and  $\| \lan x\ran^N D_{0}(\cdot)\|_\infty \lec \delta^2$
for any~$N$ provided the condition~\eqref{eq:dotpath} holds.  

\noindent Proposition~\ref{prop:spectral} in Section~\ref{sec:AppendixB} details the spectral properties of~$\calH_{0}$ for the case $\alpha(0)=1$ and $\gamma(0)=0$.
The more general case here follows   by means of the rescaling $f\mapsto \alpha f(\alpha x)$, as well as a modulation by a constant unitary matrix. 
Following the notation of Proposition~\ref{prop:spectral} one writes
\EQ{\label{eq:WGpm}
W(t) = \lambda_+(t) G_+ + \lambda_-(t) G_- + W_1(t)
}
where $\lan W_1(t), \sigma_3 G_\pm\rangle=0$ for all $t\ge0$. 
One needs to apply the aforementioned rescaling and modulation
to $G_{\pm}$ and $\mu$ with the fixed parameters $\al_{0}(0),\ga_{0}(0)$, which means that $\mu=\mu(\al_{0}(0))$, $G_{\pm}=G_{\pm}(\al_{0}(0),\ga_{0}(0))$. 
We remark that $\la_{\pm}$ as defined in~\eqref{eq:WGpm} are real-valued. Indeed, since $\| G_{\pm}\|_{2}=1$ and $G_{+} = \binom{g_{+}}{\overline{g_{+}}}=\overline{G_{-}}$, 
the Riesz projections associated with the eigenvalues $\pm i\mu$ can be seen to be
\EQ{\label{eq:Riesz}
P_{\pm} = \frac{\LR{\cdot, \sigma_{3} G_{\mp}} }{ \LR{G_{\pm}, \sigma_{3} G_{\mp}} }  G_{\pm},}
where $\LR{G_{\pm}, \sigma_{3} G_{\mp}} \in i\R\setminus\{0\}$. 
 Therefore, 
\EQ{
\lambda_{+} = \frac{\LR{ W,  \sigma_{3} G_{-}}}{ \LR{G_{\pm}, \sigma_{3} G_{\mp}} } =   \frac{ 2i \LR{ w | ig_{-}}  }{ \LR{G_{\pm}, \sigma_{3} G_{\mp}}  }  \in \R 
}
and similarly for $\la_{-}$. 

\noindent We now rewrite the $W$-equation in~\eqref{eq:driver} in the form
\EQ{\label{eq:systemparts}
&i\dot\lambda_+(t) G_+ + i\dot\lambda_-(t) G_-    - i\mu \lambda_+(t) G_+ + 
i\mu  \lambda_-(t) G_-  \\
& \qquad + i\partial_t W_1(t)  + (\calH_{0}+a_{0}(t)\sigma_3)W_1 \\
&= -D_{0}(t) W  - a_{0}(t) \sigma_3 G_+ - a_{0}(t) \sigma_3 G_- \\
&\qquad  + \dot{\gamma}(t)\xi_{0}(t) -i\dot{\alpha}(t)\eta_{0}(t) + N_{\pi_{0}}(t,W_{0}) =:F(t)
}
Denote by $P_{\pm}$, $P_0$ the Riesz projections onto $G_\pm$, and the zero root space, respectively.
Note that these operators are given by integration against exponentially decaying   tensor functions. 
 Moreover,
we write 
\EQ{\label{eq:PcP0}
 P_c = 1-P_+ -P_- - P_0 = \Preal - P_0
 }
  for the projection onto the continuous spectrum. 
Applying the projections $P_{\pm}$ to~\eqref{eq:systemparts} yields the system of ODEs
\begin{equation}\begin{aligned}
(\dot\lambda_+(t) - \mu \lambda_+(t))G_+ &= ia_{0}(t) P_+(\sigma_3 W_1) -iP_+ F(t) \\
(\dot\lambda_-(t) + \mu \lambda_-(t))G_- &= ia_{0}(t) P_-(\sigma_3 W_1) -i P_- F(t)
\end{aligned}
\label{eq:hyp_ode}
\end{equation}
For ``generic'' initial data $\lambda_{+}(0)$ the solution $\lambda_{+}(t)$ grows exponentially. However, 
there  is  a unique choice of initial condition 
that stabilizes~$\lambda_+$ (i.e., ensures that it remains bounded) leading to the
determination of the codimension one manifold. It is given by means of the following simple principle: suppose $\dot x(t) - \mu x(t) =f(t)$ with $f\in L^{\I}((0,\I))$. Then
$x\in L^{\I}((0,\I))$ iff
\EQ{
\label{eq:stable}
0 = x(0)+\int_0^\infty e^{-\mu t} f(t)\, dt.
}
Thus,  
\EQ{\label{eq:stablam+}
0= \lambda_+(0)G_+ + i\int_0^\infty  e^{-t\mu} [a_{0}(t) P_+(\sigma_3 W_1)(t) -P_+ F(t)]\, dt
}
is that unique choice. 
\eqref{eq:stablam+} has the following equivalent formulation 
\EQ{ \label{eq:la+int}
\la_{+}(t) G_{+}= - i\int_{t}^{\I} e^{-(s-t)\mu} [a_{0}(t) P_+(\sigma_3 W_1)(t) -P_+ F(t)]\, ds
}
For $\la_{-}(t)$ we have the expression
\EQ{\label{eq:la-int}
\la_{-}(t) G_{-} = e^{-\mu t}\la_{-}(0)  G_{-}+ i\int_{0}^{t} e^{-(t-s)\mu} [a_{0}(t) P_-(\sigma_3 W_1)(t) -P_- F(t)]\, ds
}
Via \eqref{eq:Riesz} one checks that $\la_{\pm}$ as defined by these equations are real-valued. 
To determine the PDE for $W_1=\Preal W_1=\Preal W$, we write  $W_1=P_c W_1 + P_0 W_1=\Wdisp + \Wroot$. Then
\begin{equation}\label{eq:mainPDE}
i\partial_t \Wdisp(t)  + (\calH_{0}+a_{0}(t)\sigma_3)\Wdisp = P_c F(t) - a(t) [\sigma_3,P_++P_-+P_0]W_1
\end{equation}
The sought after solution  
\EQ{\label{eq:sought}
(\al(t),\ga(t),\la_{+}(t),\la_{-}(t),\Wroot(t),\Wdisp(t))
} is now determined from the second and third equations of~\eqref{eq:driver}, 
from~\eqref{eq:la+int}, \eqref{eq:la-int}, and~\eqref {eq:mainPDE}. The root part is controlled by the orthogonality conditions
 \EQ{
 \LR{ W,\xi_{0}^{*}}(t) = \LR{ W, \eta_{0}^{*}}(t)=0 \quad \forall\; t\ge0
 }
 The main technical ingredient for the dispersive control of \eqref{eq:mainPDE} is the Strichartz estimate of Lemma~\ref{lem:Beclem}, see Section~\ref{sec:AppendixB}. 
 The existence of the solution~\eqref{eq:sought} is not entirely trivial since the determining equations contain these functions
 linearly on the right-hand side. However, they occur with small coefficients which allows one to iterate or contract; we skip those details. 
 The solution
obeys the estimates~\eqref{eq:dotpath}, \eqref{eq:v_disp}. 
While \eqref{eq:dotpath} has already been established in this fashion, \eqref{eq:v_disp} is obtained as follows. Assuming again~\eqref{eq:C0C1}, one concludes 
from \eqref{eq:la+int} and~\eqref{eq:la-int} that 
\EQ{
\|\lambda_{\pm}\|_{L^{\I}\cap L^{2}}\lec \delta +  (C_{0}\,C_{1}\delta+C_{0}+C_{1}^{2})\delta^{2} \ll C_{1}\,\delta 
}
provided $\delta$ is sufficiently small. 
Via Lemma~\ref{lem:Beclem} we conclude that 
$
\|\Wdisp\|_{S}\ll C_{1}\delta
$
where $S$ is the Strichartz space in~\eqref{eq:v_disp}. Finally, now that 
 the path $(\alpha(t),\gamma(t))$ has been determined, as well as  $\lambda_{\pm}(t)$, $\Wdisp(t)$, the orthogonality 
 conditions~\eqref{eq:orthcond} determine~$\Wroot$ which also satisfies $\|\Wroot\|_{S}\ll C_{1}\delta$. 
From these estimates, we conclude~\eqref{eq:v_disp} via bootstrap as claimed. 

The manifold is determined by~\eqref{eq:stablam+}
 as a graph, {\em once a fixed point} $(\pi,W)=(\pi_{0},W_{0})$ is obtained. 
 More precisely, for {\em fixed} $(\al(0),\ga(0))$  we prescribe initial conditions $W^{(0)}\in H^{1}$, $\|W^{(0)}\|_{H^{1}}\lec \delta$  
 for~\eqref{eq:driver} such that $P_{0}W^{(0)}=0$ as well as $P_{+}W^{(0)}=0$
 where the projections are relative to $\calH(\al(0),\ga(0))$.  Such data are linearly stable. 
 The condition~\eqref{eq:stablam+} takes nonlinear corrections into account
 and modifies the data in the form
 \EQ{
 \label{eq:moddata}
W(0)= W^{(0)} + h(\pi_{0},W_{0}, W^{(0)})G_{+}
}
where $h(\pi_{0},W_{0}, W^{(0)})=\lambda_{+}(0)$ is real-valued and satisfies $|h(\pi_{0},W_{0}, W^{(0)})|\lec\delta^{2}$.  
Since $\pi(0)=\pi_{0}(0)$ by construction, once we have found a fixed point,  we can write $h(\pi_{0},W_{0}, W^{(0)})=h(\al_{0}(0),\ga_{0}(0), W^{(0)})$ where the
latter is smooth in $W^{(0)}$ in the sense of Fr\'echet derivatives. Moreover, the bound
\EQ{\label{eq:hbd}
|h(\al_{0}(0),\ga_{0}(0), W^{(0)})| \lec \|W^{(0)}\|_{H^{1}}^{2}
}
will hold. This shows that \eqref{eq:moddata} describes a codimension-three manifold which is smoothly parametrized by $W^{(0)}$ and tangent
to the subspace of linear stability.  To regain the two missing codimensions, we vary $(\al_{0}(0),\ga_{0}(0))$ in a $\delta$-neighborhood of $(1,0)$. In other words, we 
let the dilation and modulation symmetries act on the codimension-three manifold. Since these symmetries act transversely on the manifold (for the
same reason that allowed us  to enforce \eqref{eq:orthcond} at $t=0$ by modifying the data), we obtain
a smooth codimension-one manifold which will be parametrized by $$(\al_{0}(0),\ga_{0}(0), W^{(0)})\in (1-\delta,1+\delta)\times (-\delta,\delta)\times B_{\delta}$$
where $B_{\delta}$ is a $\delta$-ball in $H^{1}$. This is then the sought after~$\M$. 

Thus, one needs to find a fixed point for the system \eqref{eq:driver} via
a contraction argument. The contraction argument is slightly delicate as it involves solving this system with two different but nearby {\em given}
paths $\pi_{j}^{0}$, $j=0,1$ which therefore define different Hamiltonians via~\eqref{eq:Ht}, and therefore also
different orthogonality conditions~\eqref{eq:orthcond}. Note that phases of the form $t\al^{2}$ and $t\tilde\al^{2}$ diverge linearly if $\al^{2}\ne \tilde\al^{2}$.
This make it necessary to employ a weaker norm than the one used in the previous stability argument, see~\eqref{eq:dotpath}, \eqref{eq:v_disp}. 

In order to carry out the comparison between two solutions, we work on the level of~\eqref{eq:Rvsystem} rather than with the aforementioned $W$-system.
Thus, consider two paths $\pi_{j}^{0}(t)=(\al^{0}_{j}(t),\ga_{j}^{0}(t))$ satisfying~\eqref{eq:dotpath} and with $\pi_{0}^{(0)}(0)=\pi_{1}^{(0)}(0)$,
and the associated equations with $Z=\Rv$
\EQ{\label{eq:Zsys01}
i\partial_t Z + \wt\calH_{j}(t)Z = \dot{\gamma}_{j}(t) \wt\xi_{j}(t) -i\dot{\alpha}_{j}(t) \wt \eta_{j}(t) + \wt N_{j}(t,v_{j}^{0},\bar v_{j}^{0})
}
for $j=0,1$, see~\eqref{eq:Rvsystem}. Here $\wt\calH_{j}$, $\wt\xi_{j}$, $\wt\eta_{j}$ and $\wt N_{j}(t,v,\bar v)$ are defined as
in \eqref{eq:calH}, \eqref{eq:wtxiwteta}, \eqref{eq:nonlinearity} but relative to the paths $\pi_{j}^{0}(t)$.   Moreover, the function $v_{j}^{0}$ 
are given and satisfy~\eqref{eq:v_disp}, 
and we impose the orthogonality conditions, see \eqref{eq:orthcond}, 
\EQ{\label{eq:Zorthcond}
\LR{Z(t),\sigma_{3}\wt\xi_{j}(t)} = \LR{Z(t),\sigma_{3}\wt\eta_{j}(t)}=0\qquad \forall\; t\ge0
}
The initial conditions for the paths are $\pi_{0}(0)=\pi_{1}(0)=\pi_{0}^{(0)}(0)$, whereas for $Z_{0}, Z_{1}$ one invokes~\eqref{eq:moddata} as follows:
fix $Z_{0}^{(0)}\in B_{\delta}(0)$ so that $P_{0}Z_{0}^{(0)}= P_{+}Z_{0}^{(0)}=0$ and set
\EQ{\label{eq:Z0Z1init}
Z_{0}(0) &= Z^{(0)}_{0} + h(\pi_{0}^{0},W_{0}^{0}, W_{0}^{(0)})G_{+} \\
Z_{1}(0) &= Z^{(0)}_{0} + h(\pi_{1}^{0},W_{1}^{0}, W_{0}^{(0)})G_{+} 
}
This choice guarantees that \eqref{eq:Zorthcond} holds at $t=0$. 
By the preceding stability analysis, \eqref{eq:Zsys01} and~\eqref{eq:Zorthcond} then define unique solutions $(\pi_{j},Z_{j})$ satisfying \eqref{eq:dotpath} and~\eqref{eq:v_disp}. 
Differentiating~\eqref{eq:Zorthcond} in combination with~\eqref{eq:Zsys01} yields the modulation equations~\eqref{eq:modeq}. Thus, we rewrite~\eqref{eq:Zsys01}
in the form
\EQ{\label{eq:Zsys01*}
i\partial_t Z_{j} + \wt\calH_{j}(t) Z_{j} = -iL_{j}(t)Z_{j}+ N_{j}(t,v_{j}^{0},\bar v_{j}^{0})
}
where $N_{j}$ incorporates both~$\wt N_{j}$ and the nonlinear term in~\eqref{eq:modeq}.  The linear term $ L_{j}(t)Z_{j}$  is of finite rank and corank, and
satisfies the estimates 
\EQ{
\| L_{j}(t)Z_{j} \|_{W^{k,p}}  \lec \|Z_{j}\|_{H^{1}} |\dot\pi^0_j(t)|
}
for any $1\le p\le \I$ and $k\ge0$. Combining this pointwise in time bound with~\eqref{eq:dotpath} yields the full estimates on~$L_j(t) Z_j$. 
By construction, any solution of~\eqref{eq:Zsys01*} which satisfies~\eqref{eq:Zorthcond}
at one point, say $t=0$, satisfies~\eqref{eq:orthcond} for all $t\ge0$. 

\noindent 
The difference $R:= Z_{1}-Z_{0}$ satisfies 
\EQ{\label{eq:Req}
i\partial_t R+ \wt\calH_{0}(t) R &=- i L_{1}(t) R  - N_{0}(t,v_{0}^{0},\bar v_{0}^{0}) + N_{1}(t,v_{1}^{0},\bar v_{1}^{0}) \\
&\quad + (\wt\calH_{0}(t)-\wt\calH_{1}(t))Z_{1}  - i(L_{1}(t)-L_{0}(t))Z_{0} =: \wt F
}
whereas the difference of the paths $\pi =\pi_{1} -\pi_{0}$ is governed by taking differences of the third and fourth equations, respectively,  in~\eqref{eq:driver} for $j=1,0$. 
We estimate $(R,\pi)$ in the norm, with $\rho>0$ small, fixed, and to be determined,  
\EQ{\label{eq:Ynorm}
\| (R,\pi) \|_{Y}:=  \| e^{-t\rho} R \|_{L^{\I}_{t}((0,\I);L^{2})} + \| e^{-t\rho} \dot \pi \|_{L^{1}((0,\I))}
}
To render this  a norm, one fixes $\pi(0)=0$, say. Note that some measure of growth has to be built into $\|\cdot\|_{Y}$, since 
$\wt\calH_{0}(t)-\wt\calH_{1}(t)$ and $L_{1}(t)-L_{0}(t)$ grow linearly in~$t$. 
\noindent  Next,  we perform the same modulation as above, i.e., 
$$ W(t):= 
\left[ \begin{matrix} 
 e^{-i\theta_{0}^{0}(t)} & 0 \\ 0 & e^{i\theta_{0}^{0}(t)}
 \end{matrix} \right] R,\qquad   \theta_{0}^{0}(t) = -\int_{0}^{t} (\al_{0}^{0}(s))^{2}\, ds
 $$
 Denoting the matrix here by $M_{0}(t)$, 
 $W$ satisfies the equation 
\EQ{\label{eq:Wdiffeq}
i\partial_t W+ (\wt\calH_{0}(t) + (\al_{0}^{0}(0))^{2}\sigma_{3}) W  =  M_{0} \wt F
}
To  obtain estimates on~\eqref{eq:Wdiffeq}, we write 
\[
\wt\calH_{0}(t) + (\al_{0}^{0}(0))^{2}\sigma_{3} = \calH_{0}^{0} + a(t) \sigma_3 + D(t)
\]
with the constant coefficient operator  $\calH_{0}^{0}=\calH(\al_{0}^{0}(0),\ga_{0}^{0}(0))$, see~\eqref{eq:H0}, 
and $a(t)=(\al_{0}^{0}(t))^2-(\al_{0}^{0}(0))^2 $, as well as  $D(t)$ equaling
\[
\left[ \begin{matrix} 
 - 2(Q^2(\cdot,\alpha_{0}^{0}(t)) -Q^2(\cdot,\alpha_{0}^{0}(0))) \!\!\!\! &
 \!\!\!\! -  e^{2i\gamma_{0}^{0}(t)}Q^2(\cdot,\alpha_{0}^{0}(t)) +  e^{2i\gamma_{0}^{0}(0)} Q^2(\cdot,\alpha_{0}^{0}(0)) \\
  e^{-2i\gamma_{0}^{0}(t)}  Q^2(\cdot,\alpha_{0}^{0}(t)) -   e^{-2i\gamma_{0}^{0}(0)} Q^2(\cdot,\alpha_{0}^{0}(0))\!\!\!\!  & 
  \!\!\!\!   2(Q^2(\cdot,\alpha_{0}^{0}(t))-Q^2(\cdot,\alpha_{0}^{0}(0)))
 \end{matrix} \right]
\]
One has  $\|a(\cdot)\|_\infty\lec \delta^2$ and  $\| \lan x\ran^N D(\cdot)\|_\infty \lec \delta^2$
for any~$N$ as before.   At this point the analysis is similar to the one starting with~\eqref{eq:WGpm}. Indeed, writing once again
$$
W    =     \la_{+}G_{+}    +     \la_{-}G_{-}    +    \Wroot  +   \Wdisp
$$
where the decomposition is carried out relative to~$\calH_{0}^{0}$, one inserts this into~\eqref{eq:Wdiffeq} and proceeds as before.
The two main differences from the previous stability analysis are as follows: (i)  the stability condition~\eqref{eq:stablam+} holds automatically here, since we know apriori that $\la_{+}$ remains bounded;
indeed, we chose $Z_{1}, Z_{0}$ to each satisfy~\eqref{eq:stablam+} whence~\eqref{eq:v_disp} holds for each of these functions.
(ii) the orthogonality condition~\eqref{eq:orthcond} does not hold exactly in this form, since it is obtained by taking the difference of the orthogonality conditions 
satisfied by $Z_{1}$ and $Z_{0}$. But this is minor, since the error one generates in this fashion is contractive.

Applying the dispersive bound of Lemma~\ref{lem:Beclem} (here we need only the $L^2_x$ part)
to~$\Wdisp$ yields via a term-wise estimation of the right-hand side of~\eqref{eq:Req},
\EQ{\label{eq:twoint}
\|R(t)\|_{2} &\lec \delta e^{t\rho} \|(R_{0}^{0}- R_{1}^{0}, \pi_{0}^{0}-\pi_{1}^{0})\|_{Y} + \delta \int_{0}^{t} \|R(s)\|_{2}\, ds\\
&\quad + \delta \int_{t}^{\I} e^{-\mu(s-t)} \|R(s)\|_{2}\, ds
}
where $R_{j}^{0}=\binom{v_{j}^{0}}{\bar{v}_{j}^{0} }$ for $j=0,1$. 
Recall that the initial conditions for $R$ 
are determined by~\eqref{eq:Z0Z1init}. 
The final integral in~\eqref{eq:twoint} is a result of the $\la_{+}$ equation \eqref{eq:la+int}. 
Assuming $\mu(\al_{0}^{0}(0))>\rho \gg \delta$, Gronwall's inequality implies 
\EQ{
\sup_{t\ge0} e^{-t\rho}\|R(t)\|_{2} &\lec \de\rho^{-1}  \| (v_{0}^{0} -v_{1}^{0}, \pi_{0}^{0}- \pi_{1}^{0})\|_{Y} 
}
as desired. Next, one estimates the $\pi$ equation with initial condition $\pi(0)=0$.  The conclusion is a bound of the form
\EQ{
\| (R,\pi)\|_{Y} + | h(\pi_{0}^{0},W_{0}^{0}, W_{0}^{(0)}) -  h(\pi_{1}^{0},W_{1}^{0}, W_{0}^{(0)})  | \ll \| (v_{0}^{0} - v_{1}^{0}, \pi_{0}^{0}-\pi_{1}^{0}) \|_{Y}
}
which proves the desired contractivity. See~\cite{Bec2} for more details on these estimates. 
Hence, one has a fixed point of  \eqref{eq:driver} as well as a well-defined function $h(\pi_{0}^{0}(0),W_{0}^{(0)})$. 
This concludes the proof of the existence part.

Next, we turn to scattering. In contrast to the previous analysis, we do not  linearize 
 around $\calH(\al(0),\ga(0))$, but rather around $\H(\al(\I),\ga(\I))$.  Thus, consider 
  the system \eqref{eq:modeq}, \eqref{eq:orthcond},  \eqref{eq:la+int}, \eqref{eq:la-int}, \eqref{eq:mainPDE} with $a(t)=\al^{2}(t)-\al^{2}(\I)$, 
  $F$ defined by~\eqref{eq:systemparts},
and  $D(t)$ equaling 
\EQ{\label{eq:Ddef}
\left[ \begin{matrix} 
 - 2(Q^2(\cdot,\alpha(t)) -Q^2(\cdot,\alpha(\I))) \!\!\!\! &\!\!\!\! -  e^{2i\gamma(t)}Q^2(\cdot,\alpha(t)) +  e^{2i\gamma(\I)} Q^2(\cdot,\alpha(\I)) \\
  e^{-2i\gamma(t)}  Q^2(\cdot,\alpha(t)) -   e^{-2i\gamma(\I)} Q^2(\cdot,\alpha(\I))\!\!\!\!  & \!\!\!\!   2(Q^2(\cdot,\alpha(t))-Q^2(\cdot,\alpha(\I)))
 \end{matrix} \right]
}
Thus,  $a(t), D(t)\to 0$ as $t\to\I$.  This ensures the vanishing at $t=\I$ of the first 
three terms of $F(t)$ in~\eqref{eq:systemparts}. The fourth and fifth terms of~$F$ vanish in the $L^{1}(T,\I)$-sense
as $T\to\I$ by~\eqref{eq:dotpath}, whereas the nonlinear term $N(t,W)$ vanishes in the sense of Strichartz estimates.  
Therefore, 
\eqref{eq:la+int}, \eqref{eq:la-int} imply that 
$\la_{\pm}(t)\to0$ as $t\to\I$. 
 Hence, in view of the scattering statement in
Lemma~\ref{lem:Beclem} one has the representation in $H^{1}$
\EQ{\label{eq:scat_rep}
\Wdisp(t) &= e^{i \sigma_{3} \int_{0}^{t}(-\Delta +\al^{2}(\I) +  a(s))\, ds}  W_{\I} + o(1) \\
  & =    e^{i \sigma_{3} \int_{0}^{t}(-\Delta +\al^{2}(s))\, ds}  W_{\I} + o(1)
}
as $t\to\I$. The modulation in~\eqref{eq:vwswitch} removes the $\al^{2}(s)$ in the exponent once we return to the $v$ representation. 
Finally, by the orthogonality conditions $\Wroot(t)\to0$. 
In summary, we have obtained the desired scattering statement for $v$ in~\eqref{eq:psiv}. 

Finally, to obtain the uniqueness statement let $u(t)$ be a solution with $u(0)\in B_{\delta}(Q)$ and with the property that 
$\dist(u(t),\cS_{1})\lec \delta$ for all $t\ge0$. 
We claim that there exists a   $C^1$-curve  $(\alpha(t),\theta(t))\in (1-O(\delta),1+O(\delta))\times \R$   which achieves
\begin{equation} \begin{aligned}
 \langle u(t)-e^{i\theta(t)} Q(\cdot,\alpha(t)) | e^{i\theta(t)}Q(\cdot,\alpha(t))\rangle &=0  \\
 \langle u(t)-e^{i\theta(t)} Q(\cdot,\alpha(t)) | i e^{i\theta(t)} \p_{\al} Q(\cdot,\alpha(t))\rangle &=0
 \end{aligned}
 \label{eq:modeq2}
\end{equation}
for all $t\ge 0$, as well as 
\begin{equation}\label{eq:closeuQ}
\sup_{t\ge 0} \| u(t)-e^{i\theta(t)} Q(\cdot,\alpha(t)) \|_{H^1}\lec\delta
\end{equation}
In fact, 
by definition there is a $C^{1}$--path $\tilde\theta(t)\in  \R$  
so that 
\[
\sup_{t\ge0} \| u(t)-e^{i\tilde\theta(t)} Q(\cdot,1) \|_{H^1} \lec \delta
 \]
This shows that one can fulfill~\eqref{eq:modeq2}
up to $O(\delta)$. Next, one uses that $|\langle Q| \partial_\alpha Q\rangle|\simeq 1$ for all $\alpha\simeq1$ 
 and the inverse function theorem to show that $(\tilde\alpha\equiv1,\tilde\theta)$
 can be modified by an amount $O(\delta)$ so as to exactly satisfy~\eqref{eq:modeq2} without violating~\eqref{eq:closeuQ}. Furthermore, by chaining one concludes that this procedure yields a well-defined path $(\alpha,\theta)$
 which is $C^1$, as claimed.  
Next, define \[ \theta_0(t)=-\int_{0}^t \alpha^2(s)\, ds + \theta(0)\]
and set $\gamma=\theta-\theta_0$.    
Now write
\begin{equation}
u(t) = e^{i\theta(t)} Q(\cdot,\alpha(t)) + v(t)= e^{i\theta(t)} Q(\cdot,\alpha(t)) + e^{i\theta_0(t)}w(t)
\label{eq:u_decomp}
\end{equation}
This then allows one to rewrite \eqref{eq:modeq2} in the form
\begin{equation}\label{eq:modeq3}
\langle e^{i\gamma(t)} Q(\cdot,\alpha(t)) | w(t)\rangle=0,\qquad \langle i e^{i\gamma(t)} \p_{\al} Q(\cdot,\alpha(t))| w(t)\rangle =0
\end{equation}
As before, consider $W=\binom{w}{\bar w}$, and  perform the decomposition~\eqref{eq:WGpm}. 
Inserting \eqref{eq:u_decomp} into~\eqref{eq:Schr3} yields, cf.~\eqref{eq:Rvsystem}, 
\begin{equation}\label{eq:Rvsys2}
i\partial_t \Rv + \wt\calH(t) \Rv = (\dot{\theta}(t)+\alpha^2(t)) \wt\xi(t) -i\dot{\alpha}(t) \wt \eta(t) + \wt N(t,v,\bar v)
\end{equation}
where $\wt\calH, \wt\xi,\wt\eta$, and $\wt N$ are as in \eqref{eq:calH}, \eqref{eq:wtxiwteta}, \eqref{eq:nonlinearity}. 
Furthermore, with $W=\binom{w}{\bar w}$,  
\begin{equation}\label{eq:WHt}
i\partial_t W + \calH(t) W = \dot{\gamma}(t)\xi(t) -i\dot{\alpha}(t)\eta(t) + N(t,W)
\end{equation}
see \eqref{eq:Rvsystem2}, \eqref{eq:xieta}, \eqref{eq:NtW}. The orthogonality conditions \eqref{eq:modeq3} are of the form
\begin{equation}\label{eq:orth2}
\langle W(t), \xi^*(t)  \rangle =0,\qquad \langle W(t),  \eta^*(t)  \rangle =0
\end{equation}
which is identical with \eqref{eq:orthcond}. 
This places us in the exact same position that we started from in the existence proof. Thus, the decomposition~\eqref{eq:u_decomp}
is such that \eqref{eq:dotpath} and~\eqref{eq:v_disp} hold. 
The only difference here is that
we know apriori that $\la_{+}(t)$ is bounded. However, \eqref{eq:stable}  guarantees that therefore~\eqref{eq:stablam+}
holds which forces the solution to lie on~$\M$ as desired. 
\end{proof}

\begin{rem}  
\label{rem:extendM}
Denote the manifold constructed in Proposition~\ref{prop:mf} by $\M_{1,0}$. 
The same construction can be applied to $e^{ i\gamma} Q(x,\alpha)$  instead of $Q$ for any $\ga\in \tor=\R/\Z$ and $\alpha>0$, yielding a codimension one manifold
 in the phase space $\HH$ which we 
denote by $\M_{\alpha,\ga}$.  By the uniqueness part of Proposition~\ref{prop:mf} one concludes that 
\EQ{ \label{eq:MSdef}
\M_{\cS}:=\bigcup_{\alpha>0, \ga\in\tor} \M_{\alpha,\ga}
}
is again a smooth manifold, 
which contains all of $\cS$. 
By the proof of Proposition~\ref{prop:mf} it is smoothly 
parametrized by $(\al(0),\ga(0),W(0))$ where $P_{0}(\al(0),\ga(0)) W(0)=0$
and $P_{+}(\al(0),\ga(0)) W(0)=0$ and $W(0)$ needs to be small enough. 

It has the property that any $u_{0}\in\M_{\cS}$ leads to a solution of~\eqref{eq:Schr3} defined on $t\ge0$ which scatters to~$\cS$ 
as $t\to\I$ in the sense of Definition~\ref{def:scattoQ}. 
We emphasize that this is {\em not} the manifold $(5)\cup(7)\cup(9)$ 
appearing in Theorem~\ref{thm:main*}. Rather, that manifold is the maximal backward evolution of $\M_\cS$ under the NLS flow. Note that $\M_{\cS}$, thus extended by
the nonlinear flow, is again a manifold. 
\end{rem}

The following characterization of the stable manifolds will be needed in the proof of Theorem~\ref{thm:threshold}.  It precisely captures
the case where the radiation part (i.e., the difference between $u(t)$ and the soliton in~\eqref{eq:scattoS})  has vanishing scattering data and
is therefore uniquely captured by $\la_{-}(0)$.

\begin{cor}\label{cor:paraM}
Let $\M_{\cS}$ be as in \eqref{eq:MSdef}.  Suppose $u_{0}\in \M_{\cS}$  with $M(u_{0})=M(Q)$ forward scatters to $\cS$ in 
the sense of Definition~\ref{def:scattoQ} so that~\eqref{eq:scattoS} holds with $u_{\I}=0$. Then the solution~$u(t)$
of \eqref{eq:Schr3} with data $u_{0}$ approaches a soliton trajectory in $\cS_{1}$ exponentially fast. Moreover, the solution is uniquely characterized by $\ga_{\I}\in S^{1}$
and a real number $\la_{0}$ with $|\la_0|\lec\delta$.  The case where $u$ is an exact soliton is characterized by $\la_{0}=0$. 
\end{cor}
\begin{proof}
This follows from the construction carried out in the proof of Proposition~\ref{prop:mf}, but with $\calH_{0}=\H(\al(\I),\ga(\I))$ as the driving linear operator; see
that part of the proof  dealing with scattering. By~\eqref{eq:MEal}, $\alpha(\I)=1$.  In fact, consider the representation $$W=\la_{+}G_{+} + \la_{-}G_{-}+ \Wroot+ \Wdisp$$ relative to this choice of $\H_{0}$,
and solve the system \eqref{eq:modeq}, \eqref{eq:orthcond},  \eqref{eq:la+int}, \eqref{eq:la-int}, \eqref{eq:mainPDE} with $a(t)=\al^{2}(t)-\al^{2}(\I)$, $F$ defined by~\eqref{eq:systemparts},
and  $D(t)$ given by~\eqref{eq:Ddef}. 
For \eqref{eq:modeq} one assigns the terminal conditions $\alpha(\I)=\al_{\I}=1$, $\ga(\I)=\ga_{\I}$, for \eqref{eq:la-int} we impose the initial conditions $\lambda_{-}(0)=\la_{0}$,  
and \eqref{eq:mainPDE} is solved with scattering data $W_{\I}= 0$,
cf.~\eqref{eq:scat_rep}.   Note that $\la_{+}$ does not require any further data, see~\eqref{eq:la+int}. Similarly, $\Wroot$ is determined by~\eqref{eq:orthcond}. 
The point is that we can solve the aforementioned system for $(\al(t),\ga(t))$, and $W(t)$  satisfying \eqref{eq:dotpath}, \eqref{eq:v_disp} 
by contracting in the {\em strong} norm
\EQ{\label{eq:Ynorm*}
\| (W,\pi) \|_{Y}:=  \| e^{t\rho} W \|_{L^{\I}_{t}((0,\I);L^{2})} + \| e^{t\rho} \dot \pi \|_{L^{1}((0,\I))}
}
for suitably chosen and small $\rho>0$. Note the contrast to \eqref{eq:Ynorm}. In the case of~\eqref{eq:Ynorm} the exponentially decaying weights forced us to start from $t=0$ when
carrying out the contraction argument. In the case of~\eqref{eq:Ynorm*}, however, we can solve for $\Wdisp$ from $t=\I$ due to the exponentially growing weights.
It is essential,  though, that for $\la_{-}$ we can still start at $t=0$; this is due to the fact that  equation~\eqref{eq:la-int} contains exponentially decreasing functions (one
therefore needs $\rho<\mu$ but nothing else). 
In summary, $\pi(t)-\pi(\I)$ decreases exponentially, as do $a(t), D(t)$, $\la_{\pm}$, $\Wdisp$, $\Wroot$.  This proves the exponential approach to $\cS_{1}$. 
Since $\al^{2}(t)-\al^{2}(\I)\to0$ and $\ga(t)-\ga(\I)\to0$ at an exponential rate, $u(t)$ in fact converges to a soliton trajectory in $\cS_{1}$ exponentially fast. 
The case of an exact soliton is given by $W=0$, which the contraction argument characterizes as~$\la_{-}(0)=\la_0=0$. 
\end{proof}

\section{Proof of Theorems~\ref{thm:main*}, \ref{thm:threshold}}
\label{sec:proofmain*}

\begin{proof}[Proof of Theorem~\ref{thm:main*}] We may rescale any solution to mass one. If $u$ is trapped by $\cS_{1}$,
then provided $\e\ll\delta$, where the latter is from the previous section, one concludes from the uniqueness
part of Proposition~\ref{prop:mf} that $u\in\M_{\cS}$ for large times (see Remark~\ref{rem:extendM} for the definition of~$\M_{\cS}$). Conversely, every solution starting on $\M_{\cS}$ is trapped.
Therefore, the set $(5)\cup(7)\cup(9)$ is the maximal backward evolution of $\M_{\cS}$, see Remark~\ref{rem:extendM},
whereas $(6)\cup(8)\cup(9)$ is the maximal forward evolution of   $\overline{\M_{\cS}}$ (complex conjugate). 
If we reverse time and conjugate, then the stable and unstable modes $\la_{\pm}$ are interchanged. This means that the intersection
of the center-stable manifold as $t\to\I$ with the corresponding one for $t\to-\I$ intersect transversely in a smooth manifold of codimension two, i.e., 
the center manifold.
\end{proof}

\begin{proof}[Proof of Theorem~\ref{thm:threshold}]
Let $u$ be a solution with $M(u)=M(Q)$. By assumption, $E(u)\le E(Q)$. If $E(u)<E(Q)$, then~\cite{HolR} show that either $u$ scatters at $\pm\I$,
or blows up in finite time in both directions. Therefore, assume that $E(u)=E(Q)$. If $u$ blows up  in finite negative time, then $K(u(t))<0$ for some $t<0$.
If $u$ were to scatter at $t=+\I$, then $K(u(t))>0$ for some $t>0$. But then $K(u(t_{0}))=0$ for some $t_{0}$, which implies that $u(t)=e^{i(-t+\theta_0)} Q$ which is a contradiction.
Thus, the sets $(3)$ and $(4)$ are empty. Now suppose $u$ is trapped by $\cS$ as $t\to\I$. Then for large times $u$ needs to lie on~$\M_{1,\ga}$ for some $\ga$, 
see~\eqref{eq:MEal}; in particular, $\al_{\I}=1$.  Since $E(u)=E(Q)$, it follows from Corollary~\ref{cor:paraM} that $u$ is uniquely described by $(\ga_{\I},\la_0)$.  
 Fixing the symmetry parameter $\ga_{\I}$, we see that the solution is described by the single real-valued parameter~$\la_0$.
If $\la_0=0$, then necessarily $u(t)=e^{i(-t+\ga_{\I})}Q$. Therefore, $\la_{0}\ne0$ and the sign of this parameter uniquely determines the sign of $K(u(t))$
upon ejection as in Lemma~\ref{lem:eject}, deciding the fate of $u(t)$ for negative times. This shows that one obtains two one-dimensional manifolds which approach
soliton trajectories from $\cS_{1}$ in the $H^{1}$-norm as $t\to\I$  exponentially fast, but which either blow up in finite negative time, or scatter to zero as $t\to-\I$. Since time-translation leaves these manifolds invariant,
it follows that they have the form described in Theorem~\ref{thm:threshold}. 
\end{proof}

\appendix
\section{Some tools from scattering theory}\label{sec:AppendixA}
The NLS equation~\eqref{eq:Schr3} is subcritical relative to $H^{1}(\R^{3})$ and critical relative to $\dot H^{\frac12}(\R^{3})$.  There is a classical local well-posedness theory
for~\eqref{eq:Schr3} for data of these regularity classes, see~\cite{Caz} as well as~\cite{HolR}. We work on the level of~$H^{1}$. 
As usual, we say that $(q,r)$ is Strichartz admissible in $\R^{3}$ if 
$$
\frac2q+\frac3r=\frac32
$$
and we set, using \cite{KeelTao}, 
$$
\|u\|_{S} = \sup_{\substack{(q,r)\;\text{ admissible} \\
2\leq r \leq 6, \;  2\leq q \leq \infty}} \|u\|_{L_t^qL_x^r}
$$
The following small data scattering lemma is
standard, and we leave the proof the reader.

\begin{lem}\label{lem:simplescat}
Let $u$ be a solution  of \eqref{eq:Schr3} in $\R^{3}$ on a  time-interval $I$. If there exists $t_{0}\in I$ with $\| \nabla u(t_{0})\|_{2}\le\mu$
where $\mu$ is a constant satisfying $\mu M(u)^{1/2}\ll1$, then $u$ extends to a global solution satisfying the global Strichartz bound
\EQ{\label{eq:nablaumu}
\| \nabla u \|_{L^{\infty}_{t}L^{2}_{x}} + \|\nabla u\|_{L^{2}_{t}L^{6}_{x}} \lec \mu
}
as well as
\[
\| u \|_{L^{\infty}_{t}L^{2}_{x}} + \|u\|_{L^{2}_{t}L^{6}_{x}} \lec M(u)^{1/2} 
\]
In particular, $u$ scatters: there exists $u_{0}\in H^{1}$ with $\| u(t) - e^{-it\Delta}u_{0}\|_{H^{1}}\to0$ as $|t|\to\infty$. 
\end{lem}

For the cubic equation, one has the following version of the linear Bahouri-G\'erard profile decomposition, see~\cite{BaG}, \cite{Keraani}, as well
as Lemma~5.2 of~\cite{HolR} and Proposition~6.1 of~\cite{NakS}.  All function spaces are radial. 

\begin{prop}\label{prop:BG}
Let $\{u_n\}_{n=1}^{\infty}$ be a bounded sequence in~$H^{1}$. Then 
there exist a sequence  $\{v^j\}_{j=0}^{\infty}$ bounded in $H^{1}$, and sequences of times $t_n^j\in\R$ such that for 
any $k\ge1$ one has the following property, after replacing $\{ u_{n}\}_{n=1}^{\infty}$ by a subsequence: let 
 $\ga_n^k$ defined by
\EQ{\label{eq:linprofile}
  e^{-it\Delta} u_n = \sum_{0\le j<k} e^{-i(t+t_{n}^{j})\Delta}  v^j + \ga_n^k(t) }
we have for any $0\le j<k$, $\ga_n^k(-t_n^j) \rightharpoonup 0$ weakly in $H^{1}$ as $n\to\I$,  as well as 
\EQ{ \label{orth}
 \pt \lim_{k\to \I} \limsup_{n\to\I} \|\ga_n^k\|_{L^{p}_{t} L^{q}_{x}(\R\times\R^3)}=0,   
 \pq\lim_{n\to\I} |t_n^j-t_n^k| = \I}
 where $(p,q)$ is any pair which can be obtained by interpolation\footnote{One could use here $L^{\I}_{t}L^{p}_{x}$ for  $4>p>2$.} of (some nonzero amount of) $L^\I_tL^3_x$ with $L^{2}_{t}W^{1,6}_{x}\cap L^{\I}_{t}H^{1}_{x}$. 
 In particular, $(p,q)= (\I,3)$ as well as $(\I,4), (4,4)$ are such choices. 
  Moreover, one has the following partition of the $H^{1}$-norm:
 \EQ{ \label{H1 orth}
 \limsup_{n\to\I} \Bigl|\| u_n\|_{H^{1}}^2-\sum_{j<k}\| v^j\|_{H^{1}}^2-\| \ga_n^k\|_{H^{1}}^2 \Bigr|=0,}
 and the same holds for $L^{2}$. 
\end{prop}
\begin{proof}  
One has $u_{n}\rightharpoonup v^{0}$ in $H^{1}$ (henceforth, we pass to subsequences without further mention). 
By the compact radial imbedding $H^{1}\hookrightarrow L^{p}$ for $2<p<6$ one then has strong convergence in~$L^{4}$. 
Set $t_{n}^{0}=0$. Passing to $u_{n}-v^{0}$, we may assume that $u_{n}\rightharpoonup 0$. Clearly, 
$\gamma_{n}^{1}(t) =   e^{-it\Delta} u_{n}$
satisfies $\gamma_{n}^{1}(0)\rightharpoonup 0$ as claimed. 
If 
\EQ{\label{eq:tjn}
 \liminf_{n\to\I} \| e^{-it\Delta} u_{n} \|_{L^{\I}_{t} L^{3}_{x}} =0, }
then the process terminates. Otherwise, pick $t_{1,n}$ so that $L^{\infty}_{t}$ in~\eqref{eq:tjn} is attained at those times. 
Then $e^{-it_{1,n}\Delta} u_{n}\rightharpoonup v^{1}\ne0$ by the aforementioned compact imbedding. Since $u_{n}\rightharpoonup 0$, we must have
$|t_{1,n}|\to\infty$ as $n\to\I$. 
The process now repeats inductively in a standard way, see for example~\cite{HolR}, \cite{NakS}. 
\end{proof}

Next, one has the following perturbation lemma, cf.~Lemma~6.2 in~\cite{NakS}, and Proposition~2.3 in~\cite{HolR}. 

\begin{lem} \label{PerturLem} 
There are continuous functions $\nu_0,C_0:(0,\I)^2\to(0,\I)$ such that the following holds: 
Let $I\subset \R$ be an interval, $u,w\in C(I;H^{1})$ satisfying for some $A,B>0$ and $t_0\in I$
\EQ{ \label{asm ebd}
  \|u\|_{L^\I_t(I;H^{1})} + \| w\|_{L^\I_t(I;H^{1})}  \le A, \pq \|w\|_{L^4_t(I;L^4_x)} \le B,} 
\EQ{
 \pt\|eq(u)\|_{L^{\frac85}_t(I;L^{\frac43}_x)} 
   + \|eq(w)\|_{L^{\frac85}_t(I;L^{\frac43}_x)} + \|\ga_0\|_{L^{\frac83}_t(I;L^{4}_x)} \le \nu\le  \nu_0(A,B),}
where $eq(u):=i\p_{t }u-\De u - |u|^{2}u$ and similarly for~$w$, and $\ga_0:= e^{-i(t-t_0)\Delta}(u- w)(t_0)$.  Then we have 
\EQ{ \label{eq:gammaStrich}
  \|u-w- \ga_0\|_{L^\I_tL^2_x(I)} \le C_0(A,B)\nu,\pq \|u-w\|_{L^4_{t,x}(I)}\le C_0(A,B)\nu^{1/3}.} 
\end{lem}
\begin{proof}   
We fix a $L^2$-admissible Strichartz space $Z:=L^{\frac83}_t(I;L^4_x)$ and 
\EQ{\nn 
 \ga:=u-w, \pq e:=(i\p_t-\De)(u-w)-|u|^2u+|w|^2 w=eq(u)-eq(w).} 
There exists a partition  of $I\cap[t_{0},\infty)$ such that
\EQ{\nn 
 \pt t_0<t_1<\cdots<t_n\le\I,\pq I_j=(t_j,t_{j+1}),\pq I\cap(t_0,\I)=(t_0,t_n),
\pr \max_{0\le j<n}\|w\|_{L^{4}_{t}(I_j;L^{4}_{x})} \le \de, \pq n\le C(B,\de).}
We omit the estimate on $I\cap(-\I,t_0)$ since it is the same by symmetry. 
Let $\ga_j(t):=e^{-i(t-t_j)\Delta}\ga(t_j)$. 
Then the Strichartz estimate applied to the equations of $\ga$ and $\ga_{j+1}$ implies 
\EQ{ \label{est S'}
 \pn\|\ga-\ga_j\|_{Z(I_j)} + \|\ga_{j+1}-\ga_j\|_{Z(\R)}
 \pt\lec \| |w+\ga|^2(w+\ga)-|w|^2 w+e\|_{L^{\frac85}_t(I_j;L^{\frac43}_x)}
 \pr\lec A\de \| \ga \|_{Z(I_{j})} + A^{\frac13}\|\ga\|_{Z(I_j)}^{\frac53}+\nu,}where in the second step the H\"older inequality was used in $t$ and in $x$, together with the Sobolev $H^1_x\subset L^4_x$.  
Hence by induction on $j$ and continuity in $t$, one obtains provided $A\delta\ll1$, 
\EQ{ \label{S' iterate}
 \|\ga\|_{Z(I_j)} + \|\ga_{j+1}\|_{Z(I)}
\le (2C)^j\nu \le (2C)^n\nu\ll\de,}
provided that $\nu_0(A,B)$ is chosen small enough. Repeating the estimate~\eqref{est S'} once more, we can bound $\|\ga-\ga_0\|_{L^\I_tL^2_x}$ as well. The bound in $L^4_{t,x}$ is obtained by the interpolation 
\EQ{
 \|u\|_{L^4_{t,x}}\lec \|u\|_{L^4_t W^{1/3,3}} \lec \|u\|_{L^\I_tH^1_x}^{1/3}\|u\|_{Z}}
 and we are done. 
\end{proof}

\section{Spectral properties and linear dispersive estimates}
 \label{sec:AppendixB}
 
 We begin with a result on the spectral properties of $\calH$, see~\eqref{eq:H0} with $\alpha(0)=1$, $\gamma(0)=0$.  As usual, $Q=Q(\cdot,1)$ for simplicity. 
 We view all operators in this section as complex linear ones. 
 Then 
 \[
 \calH = \left[ \begin{matrix} 
 -\Delta + 1 - 2Q^2(\cdot,1) & -  Q^2(\cdot,1) \\
    Q^2(\cdot,1) &  \Delta - 1 + 2Q^2(\cdot,1)
 \end{matrix} \right]
 \]
 is conjugate to 
 \EQ{\label{eq:HL-L+}
\frac12 \Big[\begin{matrix} 1&1 \\ -i&i \end{matrix}\Big] \calH \Big[\begin{matrix} 1 & i \\ 1  & -i \end{matrix}\Big]
 = i\Big[\begin{matrix} 0 & L_{-} \\ -L_{+} & 0 \end{matrix}\Big] 
}
where 
 \[
L_+ = -\Delta + 1 -3Q(\cdot,1)^2,\qquad L_-=-\Delta+1 -Q(\cdot,1)^2
\]
The equality \eqref{eq:HL-L+} is to be considered as one between complex linear operators. However, it is also natural to view the left-hand side
as acting on vectors $\binom{u_{1}}{u_{2}}$ with $u_{1}, u_{2}$ real-valued. In that case the right-hand side needs to be rewritten as
\[
   \Big[\begin{matrix} L_{+} & 0  \\ 0 & L_{-}  \end{matrix}\Big] 
=\LL
\]
This is exactly the point of view taken in Section~\ref{sec:ground}, where $\LL$ is considered as a real-linear operator. 
The spectral properties of $L_+,L_-$ and especially $\calH$ are quite delicate. 
The following result summarizes
what can be obtained by a rigorous analysis, see~\cite{S}, \cite{ES},  \cite{HL},  supported
by numerics, such as \cite{DS} and \cite{MS}. 
In the work of Marzuola, Simpson~\cite{MS} numerics is used to  assist  index computations of certain quadratic forms in the spirit of the virial argument of Fibich, Merle, Raphael~\cite{FMR}.

For simplicity, we restrict ourselves to the Hilbert space\footnote{The only change to Proposition~\ref{prop:spectral} is that $\calH$ has a root-space of dimension eight rather than two.} 
 $L^2_{\mathrm{rad}}(\R^3)$ in Proposition~\ref{prop:spectral}. 

\begin{prop}\label{prop:spectral}
The essential spectrum of $\calH$  is $(-\infty,-1]\cup[1,\infty)$ and there are no imbedded eigenvalues or resonances 
in the essential spectrum, the discrete spectrum is of the form $\{0, i\mu ,-i\mu \}$
where $\mu >0$ with $\pm i\mu $ both simple eigenvalues, the root-space at $0$ is of dimension two, and the thresholds $\pm 1$ are
neither eigenvalues nor resonances. In explicit form, the root space is spanned by 
\begin{equation}\label{eq:xi0eta0}
\xi_0= \binom{Q }{-Q },\qquad \eta_0=\binom{\partial_\alpha Q }{\partial_\alpha Q }
\end{equation}
and one has $\calH\xi_0=0, \calH^2 \eta_0=0$. 
 Let  $\calH G_{\pm} = \mp i\mu  G_{\pm}$ with the normalization $\|G_{\pm}\|_{2}=1$. Then the eigenfunctions $G_{\pm}$ are
exponentially decaying and of the form $G_{\pm}=\binom{g_{\pm}}{\bar{g_{\pm}}}$.
\end{prop} 
\begin{proof}
The description of the root space of~$\calH$ goes back to~\cite{Wein1}. The imaginary spectrum was identified in~\cite{S}, and for the exponential decay of the corresponding 
eigenfunctions see~\cite{HL}. See Grillakis~\cite{Gril1}, \cite{Gril2} for more on the discrete spectrum. All these results are based on purely analytical arguments.  The fact that $\calH$ does not have embedded eigenvalues in
the essential spectrum was shown in~\cite{MS}, assisted by some numerical computations. 
Their proof also implies that there are no non--zero eigenvalues in the gap $[-1,1]$,  and that the thresholds are
not resonances. Alternatively, the latter two facts  also follow by the analytical arguments in~\cite{S} combined with the numerics in~\cite{DS}. 
\end{proof}

Next, we present a result for non-selfadjoint Schr\"odinger evolutions which originates in~\cite{Bec1} (in fact, Beceanu proves a stronger result in Lorentz spaces).  
Let $S=\bigcap_{p,q} L^p_t(\R^+,L^{q}_x)$ be the Strichartz space with $2\le p \le\infty$, $2\le q\le 6$, and $\frac{2}{p}+\frac{3}{q}=\frac32$,
 and  let $S^*$ be its dual. 

 \begin{lem}\label{lem:Beclem}
Let $a\in L^\infty(\R)$ satisfy $\|a\|_\infty<c_0$ for some small absolute constant $c_0$. 
The solution $\Psi\in C(\R; L^2(\R^3) )\cap C^1(\R;H^{-2}(\R^3) )$
of the problem
 \EQ{\label{eq:Psiateq}
 i\partial_t\Psi + \calH \Psi + a(t) \sigma_3 P_c\, \Psi=F \in S^*,\qquad \Psi(0)=\Psi_0 \in L^2(\R^3)
 }
where $P_c$ is the projection corresponding to the essential spectrum of~$\calH$, 
obeys the Strichartz estimates 
\EQ{\label{eq:Stricha}
\|P_c \Psi \|_{S} \lec \|\Psi_0\|_{L^2(\R^3)} + \|F\|_{S^*}
}
Furthermore, if $\Psi_0\in H^1$, then 
\EQ{\label{eq:Stricha1}
\|\nabla P_c \Psi \|_{S} \lec \|\Psi_0\|_{H^1(\R^3)} + \| \nabla F\|_{S^*}
}
Finally, one has scattering: there exists $\Psi_{\I}\in H^{1}$ such that 
\EQ{\label{eq:Sysscat}
 P_{c} \Psi(t) = e^{i\sigma_{3}\int_{0}^{t} (-\Delta+1+a(s))\, ds\, } \Psi_{\I} + o(1)
}
in $H^{1}$ as  $t\to\I$. 
 \end{lem} 
 \begin{proof} We follow~\cite{Bec1}. 
Clearly, the proof should be perturbative in~$a$ by nature, with $a=0$ being the nontrivial statement
 that Strichartz estimates (including the endpoint) hold for the equation
 \[
  i\partial_t\Psi + \calH \Psi =F
 \]
 However, the latter has been established by several authors, see for example \cite[Theorem~1.3]{Bec1} and~\cite{CucMiz}.
Due to the lack of any physical localization of the $a(t)$ term, the perturbative analysis is nontrivial. On the other hand,
note that any perturbation of the form $a(t) \chi(x)$ where the multiplier $\chi$ is bounded $L^6(\R^3)\to L^{\frac65}(\R^3)$ can be
taken to the right-hand side by virtue of the endpoint Strichartz estimate.

To commence with the actual argument, consider the following auxiliary equation, with arbitrary but fixed $\delta>0$, and $P_{d}=\Id-P_{c}$: 
 \EQ{\label{eq:aux}
  i\partial_t Z+ \calH P_{c} Z + i\delta P_{d} Z + a(t) \sigma_3\, P_{c}Z=F 
 }
 with data $Z(0)=\Psi_{0}$.  We claim the Strichartz estimates for general data $Z(0)$, 
 \EQ{\label{eq:Strichaux}
 \| Z \|_{S} \lec \| Z(0) \|_{L^2(\R^3)} + \|F\|_{S^*}
 }
 If so, then $\tilde Z:=P_{c}Z$ satisfies  $\tilde Z(0)=P_{c}\Psi_{0}$ and 
 \EQ{
  i\partial_t \tilde Z+ \calH \tilde  Z  + a(t) \sigma_3\, \tilde Z=  P_{c}F  + a(t)[\sigma_{3}, P_{c}] \tilde Z
 }
 which is the same as the $P_{c}$ projection of \eqref{eq:Psiateq}. Thus, $P_{c}\Psi = \tilde Z$ and~\eqref{eq:Strichaux}
 implies~\eqref{eq:Stricha}.   Let $A(t)=\int_{0}^{t} a(s)\, ds$ and write $U(t)=e^{iA(t)\sigma_{3}}$, $Z(t)=U(t)\Phi$. 
 Then~\eqref{eq:aux}
 becomes 
 \EQ{ \label{eq:PhiF1}
  i\partial_t \Phi+  U^{-1}(\calH P_{c}   + i\delta P_{d}) U \Phi = U^{-1}F + a(t) U^{-1}\sigma_3\, P_{d} U\Phi =: F_{1}
 }
 or, with $\Phi(0)=Z(0)$, 
 \EQ{\label{eq:aux_gauge}
 i\partial_t \Phi+  \calH_{0}\Phi = - U^{-1}(V  - \calH P_{d}   + i\delta P_{d}) U \Phi +  F_{1}
 }
 The matrix operators $\calH_0, V$ are defined via:
\[
 \calH_0 = \left[ \begin{matrix} 
 -\Delta + 1  &0 \\
    0 &  \Delta - 1  
 \end{matrix} \right],\qquad \calH=\calH_0+V
\]
 Choose a smooth, exponentially decaying matrix potential $V_{2}$ which is invertible and such that the operator
 \[
 V_{1}:= (V  - \calH P_{d}   + i\delta P_{d}) V_{2}^{-1}
 \]
 is bounded from $L^{p}\to L^{q}$ for any $1\le p,q \le \I$.   In other words, 
 \EQ{\label{eq:V1V2}
V_{1}V_{2} =  V  - \calH P_{d}   + i\delta P_{d}
} 
 with $V_{1}, V_{2}$ being bounded from $L^{p}\to L^{q}$ for any $1\le p,q \le \I$. By Duhamel the solution to~\eqref{eq:aux_gauge}
 is 
 \EQ{\label{eq:PhiInt}
 \Phi(t) = e^{it\calH_{0}} \Phi(0) -i \int_{0}^{t}  e^{i(t-s)\calH_{0}} [ - U^{-1}V_{1}V_{2} U \Phi +  F_{1} ](s)\, ds
 }
Applying $U(t)$ to both sides yields, since $U$ commutes with the propagator of $\calH_{0}$, 
\EQ{\label{eq:Zint1}
Z(t) = U(t)e^{it\calH_{0}} Z(0) + i \int_{0}^{t}  e^{i(t-s)\calH_{0}} [ U(t) U^{-1}(s)V_{1}V_{2} Z(s) -  U(t)F_{1}(s) ]\, ds
}
 We introduce the operators 
 \EQ{
 T_{0} F(t) &:= V_{2}\int_{0}^{t} e^{i(t-s)\calH_{0}} V_{1} F(s)\, ds \\
 \tilde T_{0} F(t) &:=    V_{2}\int_{0}^{t} e^{i(t-s)\calH_{0}} U(t)U(s)^{-1}V_{1} F(s)\, ds
 }
 By the Strichartz estimates for the free equation, $T_{0}, \tilde T_{0}$ are bounded on $L^{2}_{t,x}$. 
 By \eqref{eq:Zint1}, 
 \EQ{ \label{eq:V2Z}
 V_{2}Z= i\tilde T_{0} V_{2}Z + V_{2} U(t)e^{it\calH_{0}} Z(0) - i V_{2} U(t) \int_{0}^{t}  e^{i(t-s)\calH_{0}} F_{1}(s)\, ds
 }
 Suppose 
 \EQ{\label{eq:T0inverse}
 (\Id - i\tilde T_{0})^{-1} \::\: L^{2}_{t,x} \to  L^{2}_{t,x}
 }
 as a bounded operator. Then \eqref{eq:V2Z} implies via the endpoint Strichartz estimate, see \eqref{eq:PhiF1},  
 \EQ{
 \| V_{2} Z\|_{L^{2}_{t,x}} &\lec    \|Z(0)\|_{2} + \|F_{1}\|_{S^{*}} \lec \|Z(0)\|_{2}+ \|F\|_{S^{*}}+ c_{0}\|V_{2}Z\|_{L^{2}_{t,x}}
 }
 To pass to the final estimate we wrote $P_{d}Z=P_{d}V_{2}^{-1}V_{2}Z$ and used that $P_{d}V_{2}^{-1}$ is bounded by construction. 
 Inserting the resulting bound on $\| V_{2} Z\|_{L^{2}_{t,x}}$  back into \eqref{eq:Zint1} yields the  desired estimate~\eqref{eq:Strichaux}. 
 It therefore remains to prove~\eqref{eq:T0inverse} which will follow from 
 \EQ{\label{eq:T0inverse2}
 (\Id - iT_{0})^{-1} \::\: L^{2}_{t,x} \to  L^{2}_{t,x}
 }
 provided we can show that $\|T_{0}-\tilde T_{0}\|\ll 1$ in the operator norm on $L^{2}_{t,x}$. This, however, follows from the pointwise dispersive
 estimate on $e^{it\calH_{0}} $ which yields 
 \EQ{
  \| V_{2}e^{i(t-s)\calH_{0}} ( U(t)U(s)^{-1}-1) V_{1} F(s) \|_{2} \lec \|a\|_{\infty}^{\frac{1}{4}}\langle t-s\rangle^{-\frac54} \|F(s)\|_{2}
 }
 
Thus, we have reduced ourselves to proving \eqref{eq:T0inverse2}.  We introduce 
 \EQ{
  T_{1} F(t) &:= V_{2} \int_{0}^{t} e^{i(t-s)\calH P_{c} - (t-s)\delta P_{d}} V_{1} F(s)\, ds
  }
 As for the meaning of $T_{1}$, first note that due to commutativity, 
 \EQ{
 e^{it\calH P_{c} - t\delta P_{d}} &= e^{it\calH P_{c}} e^{ - t\delta P_{d}} \\
 &=  \big(e^{it\calH }  P_{c} + P_{d} \big) e^{ - t\delta P_{d}} =  e^{it\calH }  P_{c}  +  e^{ - t\delta P_{d}} P_{d} 
 }
 satisfies Strichartz estimates as in \eqref{eq:Stricha}, see ~\cite{Bec0}, \cite{Bec1}, as well as~\cite{CucMiz}. 
 Second,   the solution to 
 \EQ{
 i\p_{t} Z + \calH P_{c} Z + i\delta P_{d}Z=0
 }
 can be written in two ways:
 \EQ{
 Z(t) &= e^{it\calH P_{c} - t\delta P_{d}} Z(0)\\
 Z(t) &= e^{it\calH_{0}} Z(0) + i\int_{0}^{t} e^{i(t-s)\calH_{0}}  (V-\calH P_{d} + i\delta P_{d}) Z(s)\, ds
 }
  Thus, one further has
 \EQ{
 e^{it\calH P_{c} - t\delta P_{d}} Z(0) = e^{it\calH_{0}} Z(0) + i \int_{0}^{t} e^{i(t-s)\calH_{0}}  (V-\calH P_{d} + i\delta P_{d}) e^{is\calH P_{c} - s\delta P_{d}} Z(0)\, ds
 }
 Therefore, we conclude that
 \EQ{
 T_{0}T_{1} F(t) &= V_{2} \int_{0}^{t} \int_{0}^{s} e^{i(t-s)\calH_{0}}  (V-\calH P_{d} + i\delta P_{d}) e^{i(s-s_1)\calH P_{c} - (s-s_{1})\delta P_{d}} V_{1}F(s_{1})\, ds_{1} ds\\
 &= - iV_{2} \int_{0}^{t} ( e^{i(t-s_1)\calH P_{c} - (t-s_{1})\delta P_{d}} - e^{i(t-s_{1})\calH_{0}} ) V_{1} F(s_{1})\, ds_{1}
 }
 or $T_{0}T_{1}+ i(T_{1}-T_{0})=0$ which implies that
 \EQ{
 (\Id-iT_{0})(\Id+iT_{1})=\Id
 }
 On the other hand, 
 \EQ{
 T_{1}T_{0} F(t) &= V_{2} \int_{0}^{t} \int_{0}^{s}  e^{i(t-s)\calH P_{c} - (t-s)\delta P_{d}} (V-\calH P_{d} + i\delta P_{d}) e^{i(s-s_{1})\calH_{0}} V_{1}F(s_{1})\, ds_{1} ds\\
 &= iV_{2} \int_{0}^{t} \int_{s_{1}}^{t} \p_{s} \big[ e^{i(t-s)\calH P_{c} - (t-s)\delta P_{d}} \; e^{i(s-s_{1})\calH_{0}}\big]\, ds\: V_{1}F(s_{1})\, ds_{1} \\
  &= - iV_{2} \int_{0}^{t} ( e^{i(t-s_1)\calH P_{c} - (t-s_{1})\delta P_{d}} - e^{i(t-s_{1})\calH_{0}} ) V_{1} F(s_{1})\, ds_{1}
 }
  whence $T_{1}T_{0}+ i(T_{1}-T_{0})=0$ which implies that
 \EQ{
 (\Id + iT_{1})(\Id - iT_{0})=\Id
 }
 These identities hold in the algebra of bounded operators on $L^2_{t,x}$, as justified by the endpoint Strichartz estimates. 
Thus \eqref{eq:T0inverse2} holds and \eqref{eq:Stricha} follows.  For~\eqref{eq:Stricha1} one applies a gradient to~\eqref{eq:Psiateq}.

 From \eqref{eq:PhiInt}, we obtain the scattering of $\Phi$ in the following sense:
 \EQ{
 \Phi(t) = e^{it\calH_{0}} \Phi_{\I} + o(1) \quad t\to\I
 }
 in $H^{1}$ for some $\Phi_{\I}\in H^{1}$. Thus,
  \EQ{
 P_{c}\Psi(t) = P_{c}U(t)\big[e^{it\calH_{0}} \Phi_{\I} + o(1)] =  e^{iA(t)\sigma_{3}} e^{it\calH_{0}} \Phi_{\I} + o(1)  \quad t\to\I
 }
 in $H^{1}$,  as claimed.     
  \end{proof}

\section{Some radial Sobolev inequalities} \label{appendix radSob}
For the reader's convenience, we prove some elementary Sobolev-type inequalities for radial functions used in this paper. For any radial smooth function $u(x)=u(r)$ with  compact support on $\R^3$, we have for any $R>0$, 
\EQ{
 \sup_{r>R}|u(r)|^2 \le \int_R^\I|2uu_r|dr \le 2\|u_r\|_{L^2_r(R,\I)}\|u\|_{L^2_r(R,\I)},}
by Cauchy-Schwarz, where $L^2_r$ denotes the $L^2$ space for $r\in\R$ without any weight. Also by partial integration, 
\EQ{ \label{L^2>R bd}
 \int_R^\I|u(r)|^2dr \le \int_R^\I 2|uu_r(r-R)|dr \le 2\|u_r\|_{L^2(R,\I)}\|ru\|_{L^2_r(R,\I)}.}
Plugging the latter estimate into the former one obtains 
\EQ{ \label{L^I>R bd}
 \|u\|_{L^\I_r(R,\I)} \le 2\|u_r\|_{L^2_r(R,\I)}^{3/4}\|ru\|_{L^2_r(R,\I)}^{1/4}.}
Combining the above two estimates yields
\EQ{
 \int_R^\I |u(r)|^4dr \le \|u\|_{L^\I_r(R,\I)}^2\|u\|_{L^2_r(R,\I)}^2
  \le 8\|u_r\|_{L^2_r(R,\I)}^{5/2}\|ru\|_{L^2_r(R,\I)}^{3/2},}
and 
\EQ{
 \int_R^\I |u(r)|^4r^2dr \le \|u\|_{L^\I_r(R,\I)}^2\|ru\|_{L^2_r(R,\I)}^2
  \le 4\|u_r\|_{L^2_r(R,\I)}^{3/2}\|ru\|_{L^2_r(R,\I)}^{5/2}.}
Interpolating the above two, we obtain 
\EQ{ \label{rad Sob sg+}
 \int_R^\I |u(r)|^4rdr \le 4\sqrt{2}\int_R^\I|u_r(r)|^2dr \int_R^\I|u(r)|^2r^2dr.}
By a density argument this estimate extends to any radial $u\in H^1(\R^3)$. Furthermore, 
\EQ{
 \sup_{r>R}|u(r)|^2 \le \int_R^\I|2uu_r|dr \le 2R^{-2}\|ru_r\|_{L^2_r(R,\I)}\|ru\|_{L^2_r(R,\I)},}
and so 
\EQ{ \label{rad Sob sg-}
 \int_R^\I|u(r)|^4r^2dr \le \|u\|_{L^\I_r(R,\I)}^2\|ru\|_{L^2_r(R,\I)}^2
  \le 2R^{-2}\|ru_r\|_{L^2_r(R,\I)}\|ru\|_{L^2_r(R,\I)}^3.}
  
\section{Table of Notation}
{\small
\begin{longtable}{l|l|l}
 \hline 
 symbols & description & defined in \\
 \hline 
 $M(u),E(u)$ & mass and energy & \eqref{eq:MundE} \\
 $\HH$, $\HH^\e$, $\HH^{\e}_{\al}$, $\HH_{(\de)}$  & energy space and subsets &   \eqref{eq:HHe}, \eqref{eq:HHea}, \eqref{def HS} \\
 $Q$, $Q_{\al}$,  $\cS$, $\cS_{\al}$ & solitons, soliton manifolds &   Section~\ref{sec:intro}  \\
 $Q_{\al}'$ & derivative of $Q_{\al}$ in $\al$ & \eqref{eq:Q'def} \\
 \hline 
 $J(u)$, 
 $K(u)$, $G(u)$, $I(u)$ & action and derived functionals  &  \eqref{eq:EJK}, \eqref{def G}  \\
 $\LL$, $L_{+}$, $L_{-}$ & $\R$-linear  linearized Hamiltonian &    \eqref{eq:RlinLdef} \\
 $\GZ_\pm$, $\la_{\pm}$  & unstable/stable modes of $i\LL$ &  \eqref{eq:GZpm} \\
 $C(w)$ & super-quadratic part of $J(u)-J(Q)$ & \eqref{def C}\\
 $\al$, $\theta$ & modulation parameters & \eqref{eq:para} \\
 $\|\cdot\|_{E}^{2}$ & linearized energy norm & \eqref{eq:linEnorm} \\
 $d_{Q}(u)$ & nonlinear distance to $\cS_{1}$ & \eqref{eq:dQdef} \\
 \hline  
 $\delta_{E}$, $\delta_{X}$ & smallness scales for ejection & Lemma~\ref{lem:eject} \\
 $\Sg(u)$ & continuation of $\sign(K(u))$ to $\HH_{(\de)} $ & Lemma~\ref{lem:sign}\\
 $\delta_{S}$ & smallness scale needed for $\Sg(u)$ & Lemma~\ref{lem:sign}   \\
 $\e_*$,  $R_*$ & smallness scales for  $1$-pass theorem & Theorem~\ref{thm:onepass}, \eqref{cond0 Rede*} \\
  \hline
  $v_{n}^{j}$, $u_{n}^{j}$, $\ga_{n}^{k}$ & Bahouri-Gerard decomposition & Section~\ref{sec:scat}\\
  \hline
  $\calH(\al,\ga)$ & matrix Hamiltonian & \eqref{eq:H0} \\
  $P_{c}$, $P_{0}$, $P_{\pm}$ & Riesz projections for $\calH$ &  \eqref{eq:PcP0} \\
  $\M$, $\M_{\al,\ga}$, $\M_{\cS}$ & center stable manifolds & Section~\ref{sec:mf} \\ 
  \hline
  $\xi_{0}$, $\eta_{0}$ & root modes of matrix Hamiltonian & \eqref{eq:xi0eta0} \\
  $G_{\pm}$ & discrete imaginary modes of $\calH$ & Proposition~\ref{prop:spectral} \\
\end{longtable}
}

\section*{Acknowledgments}
The authors thank the referees for useful comments, which in particular simplified the virial argument drastically in the scattering region. They also thank Guixiang Xu for pointing out several misprints. 
The second author was partially supported by a Guggenheim fellowship and the National Science Foundation, DMS--0653841.

\end{document}